\documentclass[11pt]{article}
\usepackage[english]{babel}
\usepackage[latin1]{inputenc}
\usepackage{amsmath}
\usepackage{amsfonts}
\usepackage{amssymb}
\usepackage{amsthm}
\usepackage{mathrsfs}
\usepackage{mathtools}

\newtheorem{thm}{Theorem}
\newtheorem{defn}[thm]{Definition}
\newtheorem{rem}[thm]{Remark}
\newtheorem{cor}[thm]{Corollary}
\newtheorem{prop}[thm]{Proposition}
\newtheorem{lem}[thm]{Lemma}

\renewcommand{\geq}{\geqslant}

\begin{document}

\title{\bf{\LARGE{{Explicit expression for the generating function\\counting Gessel's walks}}}}

\author{Irina Kurkova\footnotemark[1]
        \and
        Kilian Raschel\footnotemark[1] }

\date{October 5, 2010}

\maketitle

\footnotetext[1]{Laboratoire de Probabilit\'es et
        Mod\`eles Al\'eatoires, Universit\'e Pierre et Marie Curie,
        4 Place Jussieu, 75252 Paris Cedex 05, France.
        E-mails: \texttt{irina.kourkova@upmc.fr},
                 \texttt{kilian.raschel@upmc.fr}}

\vspace{-6mm}

\begin{abstract}
Gessel's walks are the planar walks that move within the positive quadrant $\mathbb{Z}_{+}^{2}$ by unit steps in any of the following directions: West, North-East, East and South-West. In this paper, we find an explicit expression for the trivariate generating function counting the Gessel's walks with $k\geq 0$ steps, which start at $(0,0)$ and end at a given point $(i,j) \in \mathbb{Z}^2_+$.
\end{abstract}

\vspace{3mm}


\vspace{1.5mm}

\noindent {\it AMS $2000$ Subject Classification: primary 05A15; secondary 30F10, 30D05}

\section{Introduction}
\label{Introduction}

The enumeration of lattice walks is a classical problem in
combinatorics and this article is about the special case of Gessel's walks.
These are the planar walks that move within the positive quadrant $\mathbb{Z}_{+}^{2}$ by unit steps in any of the following directions: West, North-East, East and South-West. To be more
precise, let the walk be of length $k\geq 0$ and end at $(i,j) \in \mathbb{Z}_{+}^{2}$:
if $i,j>0$ then the next step can either be at $(i-1,j)$, $(i+1,j+1)$,
$(i+1,j)$ or $(i-1,j-1)$; if $i>0$ and $j=0$
it can be at $(i-1,0)$, $(i+1,1)$ or $(i+1,0)$; if $i=0$ and $j\geq0$ it can be
at $(1,j+1)$ or $(1,j)$. This is illustrated on
Figure~\ref{Steps_Gessel} below.

\begin{figure}[!ht]
\begin{center}
\begin{picture}(65.00,65.00)
\includegraphics{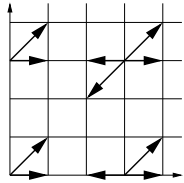}
\end{picture}
\end{center}
\vspace{-3mm}
\caption{Steps in Gessel's walks}
\label{Steps_Gessel}
\end{figure}
For $(i,j) \in \mathbb{Z}_{+}^{2}$ and $k\geq 0$, let
     \begin{equation*}
          q(i,j,k)=\#\big\{\text{Gessel's walks of length } k \text{ starting at}\ (0,0)\ \text{and
          ending at}\ (i,j)\big\}.
     \end{equation*}
%Gessel's walks have puzzled the mathematics community since 2001,
%when Ira Gessel~conjectured that for any $k\geq 0$,
%$q(0,0,2k)=16^{k} \big[(5/6)_{k} (1/2)_{k}\big]\big/\big[(2)_{k}
%(5/3)_{k}\big]$, where $(a)_{k}=a(a+1)\cdots (a+k-1)$. Only in
%$2009$ did Kauers, Koutschan and Zeilberger \cite{KKZ} give a proof
%of this conjecture.
% POUR EVITER TOUTE POLEMIQUE POURQUOI NE PAS ECRIRIRE DE MANIERE PLATE SANS SAVEUR ?
These walks are named after Ira Gessel, who
conjectured in 2001 that for any $k\geq 0$, $q(0,0,2k)=16^{k} \big[(5/6)_{k}
(1/2)_{k}\big]\big/\big[(2)_{k} (5/3)_{k}\big]$, where
$(a)_{k}=a(a+1)\cdots (a+k-1)$. This was proven in 2009 by Kauers, Koutschan and
Zeilberger \cite{KKZ}.

Let $Q(x,y,z)$ be the trivariate generating function counting Gessel's walks:
\begin{equation*}
     \label{definition_generating_function}
          Q(x,y,z)=\smashoperator{\sum_{i,j,k\geq 0}}q(i,j,k)x^{i}y^{j}z^{k}.
     \end{equation*}
     This series is entirely characterized
     by the generating functions $Q(x,0,z)$, $Q(0,y,z)$
and $Q(0,0,z)$---which count the walks that end on the borders of the
quadrant---via the following functional equation:
     \begin{equation}
     \label{functional_equation}
x y z \left(\frac{1}{x}+\frac{1}{x y}+x + x y -\frac{1}{z}\right)
Q(x,y,z)=zQ(x,0,z)+z(y+1)Q(0,y,z)-zQ(0,0,z)-x y.
     \end{equation}
This equation, a classical result \cite{BK,BMM}, is valid \text{a priori}
in the domain $\{(x,y,z): \ |x|\leq 1,|y|\leq 1, |z|<1/4\}$.

  In this article, with the help of complex analysis methods we obtain closed expressions for
   $Q(x,0,z)$, $Q(0,y,z)$ and $Q(0,0,z)$ from
  exploiting the functional equation
(\ref{functional_equation}).

By using computer calculations, Bostan and Kauers \cite{BK} showed that
$Q(x,y,z)$~is algebraic and found the minimal polynomials of
$Q(x,0,z)$, $Q(0,y,z)$ and~$Q(0,0,z)$---by ``minimal polynomial'' of
an algebraic function $F(u,z)$, we mean the unique monic polynomial
with coefficients in $\mathbb{C}[u,z]$ dividing any polynomial which
vanishes at $F(u,z)$. Van Hoeij \cite{BK} then managed to compute $Q(x,0,z)$,
$Q(0,y,z)$ and $Q(0,0,z)$ as their roots.
% and via (\ref{functional_equation}) thus
%also $Q(x,y,z)$.

Furthermore, Gessel's model is one of the $2^8$ models of walks in
the quarter plane $\mathbb{Z}_{+}^{2}$ starting from the origin and
allowing steps in the interior according to a given subset of
$\{-1,0,1\}^{2} \setminus \{(0,0)\}$. Bousquet-M\'elou and Mishna
\cite{BMM} provided a systematic analysis of the enumeration of
these walks.  After eliminating trivial models and those already
solved \cite{MBM2} by reduction to walks in a half-plane (a simpler problem),
$79$ inherently different models remained. Following an idea
of \cite{FIM}, they associated each of them with  a group of birational
transformations of $\mathbb{C}^{2}$---for details about this group see
Subsection~\ref{Subsection_Uniformization} of this article. The
group is finite in $23$ cases and infinite in the $56$ others.
Through functional equations analogous to
(\ref{functional_equation}), they found the generating functions
%counting walks that end on the borders
for $22$ of the models with an
underlying finite group. The following property has been
crucial to their analysis: in the (half-)orbit sum of the trivariate
generating function, all terms, except for the one corresponding to
the identity element of the group, have a positive part in $x$ or in
$y$ equal to zero. This last property is not valid for Gessel's
walks and consequently, Gessel's model is the only one with a finite
group that is not solved in \cite{BMM}.

 Our method of applying (\ref{functional_equation}) is notably different and can be generalized
to all $79$ walks described above \cite{groupinfinite}. It heavily relies on the analytic approach
developed by Fayolle, Iasnogorodski and Malyshev \cite{FIM} and
proceeds by reduction to boundary value problems (BVPs) of Riemann-Carleman
type.
    In the rest of the introduction we sketch this approach
 and explain how we adapt it to the enumeration of Gessel's walks.

    The authors of \cite{FIM} find explicit expressions for the generating functions of
the stationary distributions $(\pi_{i,j})_{i,j\geq 0}$ of some
ergodic random walks in the positive quadrant $\mathbb{Z}_{+}^{2}$,
supposed to have four domains of spatial homogeneity: the interior
$\{(i,j): i,j>0\}$, the real axis $\{(i,0): i>0\}$, the imaginary
axis $\{(0,j): j>0\}$ and the origin $\{(0,0)\}$. In the interior,
the only possible non-zero jump probabilities correspond to the
eight nearest neighbors.

First, they reduce the problem to solving the following functional
equation on $\{(x,y): \ |x|\leq 1, |y|\leq 1\}$:
  \begin{equation}
  \label{efim}
       K(x,y)\Pi(x,y)=k(x,y)\pi(x)+\widetilde{k}(x,y)\widetilde{\pi}(y)+
       k_{0}(x,y)\pi_{0,0},
  \end{equation}
where the polynomials $K(x,y)$, $k(x,y)$, $\widetilde{k}(x,y)$ and
$k_{0}(x,y)$ are known, while
%$K(x,y)$ being a second degree polynomial in $x$ and $y$,
the functions $\Pi(x,y)=\sum_{i,j\geq 1}\pi_{i,j}x^{i-1}y^{j-1}$,
$\pi(x)=\sum_{i\geq 1}\pi_{i,0}x^{i-1}$ as well as $\widetilde{\pi}(y)=\sum_{j\geq
1}\pi_{0,j}y^{j-1}$ are unknown but holomorphic in their unit disc;
the constant $\pi_{0,0}$ is unknown as well.

Second, they  continue the functions $\pi(x)$ and
$\widetilde{\pi}(y)$ meromorphically (with poles that can be
identified) to the whole complex plane cut along some segments, see
Chapter~3 of \cite {FIM}.

Third, they prove that both unknown meromorphic functions $\pi(x)$
and $\widetilde{\pi}(y)$ satisfy boundary value conditions of
Riemann-Carleman type, see (5.1.5) on page 95 of \cite{FIM}. Using
information on the poles of $\pi(x)$ and $\widetilde{\pi}(y)$,
  they reduce the problems to finding some new holomorphic
   functions that are solutions to BVPs of the same type, see pages 119--124, and
   particularly (5.4.10).
If the index $\chi$ of these BVPs is non-negative (which actually is
the generic situation considered in \cite{FIM}), their solutions are
not unique but depend on $\chi+1$ arbitrary constants (for  the
notion of the index see (5.2.7) on page 98 or (5.2.42) on page 108).
Consequently, in Part 5.4, the authors of \cite{FIM} reduce the BVPs
above to finding holomorphic functions satisfying some new BVPs of
Riemann-Carleman type with an index $\chi=-1$. Finally, these last
problems are uniquely solved by converting them into BVPs of
Riemann-Hilbert type, see Theorem~5.2.8 on page 108 of \cite{FIM}.
The functions $\pi(x)$ and $\widetilde{\pi}(y)$  can be
reconstructed from these solutions; the constant $\pi_{0,0}$ is
computed from the fact that $\sum_{i,j\geq 0}\pi_{i,j}=1$.

%% J'AI REPRIS LA PARAGRAPG PRECEDENT, PAR CE QUE  C'ETAIT FAUX

%poles that can be identified) up to the whole complex plane cut
%along some segments. This ingenious continuation procedure is the
%crucial step of \cite{FIM}.

%\medskip

Compared to (\ref{efim}), our equation (\ref{functional_equation})
seems somewhat more difficult to analyse, since~it involves an additional
parameter $z$. On the other hand, the unknowns
$zQ(x,0,z)$,~$z(y+1)Q(0,y,z)$ and $z Q(0,0,z)$
of (\ref{functional_equation}) have constant coefficients,
unlike $\pi(x)$, $\widetilde{\pi}(y)$ and $\pi_{0,0}$~in (\ref{efim}).
% $k(x,y)$, multiplied by
%$\widetilde{k}(x,y)$ and $k_{0}(x,y)$.
%As we
%will see later,
This fact implies (see Subsection \ref{CONTINUATION}) that $zQ(x,0,z)$ and $z(y+1)Q(0,y,z)$~can~be
continued  to whole cut complex planes as
holomorphic and not only meromorphic functions.
 It
also entails that $zQ(x,0,z)$ and $z(y+1)Q(0,y,z)$ satisfy BVPs of
Riemann-Carleman type with an index $\chi=0$, whose solutions are
unique, up to additive constants.
  Then, unlike \cite{FIM},
 we don't transform these problems anymore
  but we solve them directly.
 Their solutions uniquely determine
  the functions $zQ(x,0,z)-zQ(0,0,z)$ and
$z(y+1) Q(0,y,z)-z Q(0,0,z)$. The quantity $Q(0,0,z)$  is then
found easily, e.g.\ from (\ref{functional_equation}) by making the substitution
$(x,y,z)=(0,-1,z)$---which is such that the left-hand
side of (\ref{functional_equation}) vanishes---see also
Remark~\ref{sdf}. Finally, $Q(x,y,z)$ is determined via
(\ref{functional_equation}).

%Compared to (\ref{efim}), our equation (\ref{functional_equation})
%seems a bit more difficult to analyse, as it involves an additional
%parameter $z$. On the other hand, the coefficients $k(x,y)$,
%$\widetilde{k}(x,y)$ and $k_{0}(x,y)$ in front of the unknowns
%$zQ(x,0,z)$, $z(y+1)Q(0,y,z)$ and $z Q(0,0,z)$ are absent. This will
%allow us to continue $zQ(x,0,z)$ and $z(y+1)Q(0,y,z)$ as holomorphic
%and not only meromorphic functions. Furthermore, due to this fact,
%these functions satisfy boundary value problems of Riemann-Carleman
%type that belong to a special class of Carleman-Dirichlet problems.
%The index of these problems being zero, their solutions
%are unique,
% up to an additive constant.
%  Contrary to \cite{FIM},
% we shall not transform these problems anymore
%  but solve them directly.
% Their solutions determine
% uniquely the functions $zQ(x,0,z)-zQ(0,0,z)$ and
%$z(y+1) Q(0,y,z)-z Q(0,0,z)$.
% The function $Q(0,0,z)$ is then found easily
%by substituting in (\ref{functional_equation}) any triplet $(\hat
%x,\hat y, \hat z)$ such that $L(\hat x, \hat y, \hat z)=0$,
% e.g.\ $(0,-1,z)$, see Remark~\ref{sdf} above. Finally
%$Q(x,y,z)$ is determined via (\ref{functional_equation}).

\section{Reduction  to  boundary  value  problems of Riemann-Carle- man type}
\label{reduc}

\noindent{\bf Assumption.}
 In the sequel, we will suppose  that $z$ is fixed in
$]0,1/4[$.

\medskip

Before we state our main results, we must have a closer look at the
kernel
$
 L(x,y,z)=
x y z \big[1/x+1/(x y)+x + x y -1/z\big]
$
 that appears in (\ref{functional_equation}) and introduce
some notations. The polynomial $L(x,y,z)$ can be written as
     \begin{equation*}
          L(x,y,z)=\widetilde{a}(y,z)x^{2}+\widetilde{b}(y,z)x+
          \widetilde{c}(y,z)=a(x,z)y^{2}+b(x,z)y+c(x,z),
     \end{equation*}
where
%     \begin{align*}
%          \widetilde{a}(y,z)&=z y (y+1),& \widetilde{b}(y,z)&=-y, &\widetilde{c}(y,z)&=z(y+1),\\
%          a(x,z)&=z x^{2},& b(x,z)&=z x^{2}-x+z,&c(x,z)&=z.
%     \end{align*}
$\widetilde{a}(y,z)=z y (y+1)$,
$\widetilde{b}(y,z)=-y$, $\widetilde{c}(y,z)=z(y+1)$ and $a(x,z)=z x^{2}$, $b(x,z)=$ $z x^{2}-
x+z$, $c(x,z)=z$.
Define also
%$\widetilde{d}(y,z)=\widetilde{b}(y,z)^{2}-4\widetilde{a}(y,z)
%\widetilde{c}(y,z)$ and $d(x,z)=b(x,z)^{2}-4a(x,z)c(x,z)$.
     \begin{equation*}
       \widetilde{d}(y,z)=\widetilde{b}(y,z)^{2}-4\widetilde{a}(y,z)
          \widetilde{c}(y,z),\ \ \ \ \ d(x,z)=b(x,z)^{2}-4a(x,z)c(x,z).
     \end{equation*}

We have $L(x,y,z)=0$ if and only if $[\widetilde{b}(y,z)+
2\widetilde{a}(y,z)x]^{2}=\widetilde{d}(y,z)$ or equivalently
$[b(x,z)+2a(x,z)y]^{2}=d(x,z)$. In particular, the algebraic
functions $X(y,z)$ and $Y(x,z)$ defined by $L(X(y,z),y,z)=0$ and
$L(x,Y(x,z),z)=0$ respectively have two branches, namely
     \begin{align*}
          X_{0}(y,z)&=[-\widetilde{b}(y,z)+ \widetilde{d}(y,z)^{1/2}]
          /[2\widetilde{a}(y,z)],& X_{1}(y,z)&=[-\widetilde{b}(y,z)-
          \widetilde{d}(y,z)^{1/2}]/[2\widetilde{a}(y,z)],\\
          Y_{0}(x,z)&=[-b(x,z)+d(x,z)^{1/2}]/[2a(x,z)],
          & Y_{1}(x,z)&=[-b(x,z)-d(x,z)^{1/2}]/[2a(x,z)].
     \end{align*}

%\textit{From now on, we are going to suppose that $z$ is a f fixed,
%$z\in ]0,1/4[$}.
For any $z\in]0,1/4[$, the polynomial $\widetilde{d}$ has one root equal to zero, say
$y_{1}(z)=0$, as well as two real positive roots, that we denote by
$y_{2}(z)=[1-8z^{2}-(1-16z^{2})^{1/2}]/[8z^{2}]$ and
$y_{3}(z)=[1-8z^{2}+ (1-16z^{2})^{1/2}]/[8z^{2}]$.
 We have
$0<y_{2}(z)<1<y_{3}(z)$. We also note $y_{4}(z)=\infty$. The points
$y_{k}(z)$, $k\in\{1,\ldots ,4\}$ are the four branch points of the
algebraic function $X(y,z)$.

% Intervertir les paragraphes

Likewise, for all $z\in ]0,1/4[$, $d$ has four real positive roots,
that we denote by $x_{1}(z)=[1+2z-(1+4z)^{1/2}]/[2z]$,
$x_{2}(z)=[1-2z-(1-4z)^{1/2}]/[2z]$,
$x_{3}(z)=[1-2z+(1-4z)^{1/2}]/[2z]$ and
$x_{4}(z)=[1+2z+(1+4z)^{1/2}]/[2z]$. We have
$0<x_{1}(z)<x_{2}(z)<1<x_{3}(z)<x_{4}(z)$. The points $x_{k}(z)$,
$k\in\{1,\ldots ,4\}$ are the four branch points of the algebraic
function $Y(x,z)$.

We now present some properties of the two branches of both
$X(y,z)$ and $Y(x,z)$.

     \begin{lem}
     \label{Properties_X_Y_0}
          The following properties hold.
\begin{enumerate}
\item $X_{k}(y,z)$, $k\in\{0,1\}$ are meromorphic on
$\mathbb{C}\setminus([y_{1}(z),y_{2}(z)]\cup
[y_{3}(z),y_{4}(z)])$.  On the latter domain,
          $X_{0}$ has a simple zero at $-1$, no other zero and no pole; $X_{1}$ has a
          simple pole at $-1$, no other pole and no zero. Moreover, both $X_{0}$ and $X_{1}$
          become infinite at $y_{1}(z)=0$ and zero at $y_{4}(z)=\infty$.

\item For all $y\in\mathbb{C}$, we have $|X_{0}(y,z)|\leq
|X_{1}(y,z)|$.

\item $Y_{k}(x,z)$, $k\in\{0,1\}$ are meromorphic on
$\mathbb{C}\setminus([x_{1}(z),x_{2}(z)]\cup
[x_{3}(z),x_{4}(z)])$.
          On the latter domain, $Y_{0}$ has a double zero at $\infty$, no other zero and no pole; $Y_{1}$ has
          a double pole at $0$, no other pole and no zero.

\item For all $x\in\mathbb{C}$, we have $|Y_{0}(x,z)|\leq
|Y_{1}(x,z)|$.
\end{enumerate}
     \end{lem}

%The following result is illustrated on Figure~\ref{Locations_of_the_curves}.

Neither $X_{k}(y,z)$, $k\in\{0,1\}$ is defined for $y$ in the
branch cuts $[y_{1}(z),y_{2}(z)]$ and $[y_{3}(z),y_{4}(z)]$.
However, the limits $X_{k}^{\pm} (y,z)$ defined by
$X_{k}^{+}(y,z)=\lim X_{k}(\hat{y},z)$ as $\hat{y}\to y$ from the
\textit{upper} side of the cuts and $X_{k}^{-}(y,z)=\lim
X_{k}(\hat{y},z)$ as $\hat{y} \to y$ from the \textit{lower} side of
the cuts are well defined. These two quantities are each other's
complex conjugate, since for $y$ in the branch cuts,
$\widetilde{d}(y,z)<0$. A similar remark holds for $Y_{k}(x,z)$,
$k\in\{0,1\}$ and $x$ in $[x_{1}(z),x_{2}(z)]$ or
$[x_{3}(z),x_{4}(z)]$. Precisely, for $y\in[y_{1}(z),y_{2}(z)]$ and
$x\in[x_{1}(z),x_{2}(z)]$, we have
     \begin{equation}
     \label{values_cuts}
          X_{0}^{\pm}(y,z)=\frac{-\widetilde{b}(y,z)\mp \imath
          \big[-\widetilde{d}(y,z)\big]^{1/2}}{2\widetilde{a}(y,z)},
          \ \ \ \ \ Y_{0}^{\pm}(x,z)=\frac{-b(x,z)\mp \imath
          \big[-d(x,z)\big]^{1/2}}{2a(x,z)},
     \end{equation}
$X_{1}^{\pm}(y,z)=X_{0}^{\mp}(y,z)$ and
$Y_{1}^{\pm}(x,z)=Y_{0}^{\mp}(x,z)$. Furthermore, the identities
(\ref{values_cuts}) are true for $y\in[y_{3}(z),y_{4}(z)]$ and
$x\in[x_{3}(z),x_{4}(z)]$ respectively if one exchanges
$X_{0}^{\pm}(y,z)$ in $X_{0}^{\mp}(y,z)$ and $Y_{0}^{\pm}(x,z)$ in
$Y_{0}^{\mp}(x,z)$.

In order to state our main results, we also need the
following lemma.
     \begin{lem}
     \label{Properties_curves_0}
        The curves $X([y_{1}(z),y_{2}(z)],z)$ and
        $Y([x_{1}(z),x_{2}(z)],z)$ satisfy the following properties.
     \begin{enumerate}
     \item      They are symmetrical w.r.t.\ the real axis and
                not included in the unit disc.

     \item      $X([y_{1}(z),y_{2}(z)],z)$ contains $\infty$ and
                $Y([x_{1}(z),x_{2}(z)],z)$ is closed.

     \item \label{incc}
                They split the plane $\mathbb{C}$ in two connected
                components.
                We call $\mathscr{G}X([y_{1}(z),y_{2}(z)],z)$
                and $\mathscr{G}Y([x_{1}(z),x_{2}(z)],z)$
                the ones of the
                 branch points
                $x_{1}(z)$ and $y_{1}(z)$ respectively. They are such that
                $[x_{1}(z),x_{2}(z)]\subset \mathscr{G}X([y_{1}(z),y_{2}(z)],z)
                \subset\mathbb{C}\setminus [x_{3}(z),x_{4}(z)]$
                as well as $[y_{1}(z),y_{2}(z)]\subset\mathscr{G}Y([x_{1}(z),x_{2}(z)],z)
                \subset \mathbb{C}\setminus [y_{3}(z),y_{4}(z)]$.
\end{enumerate}
     \end{lem}

\begin{figure}[!ht]
\begin{center}
\begin{picture}(000.00,740.00)
\hspace{-117mm}\includegraphics{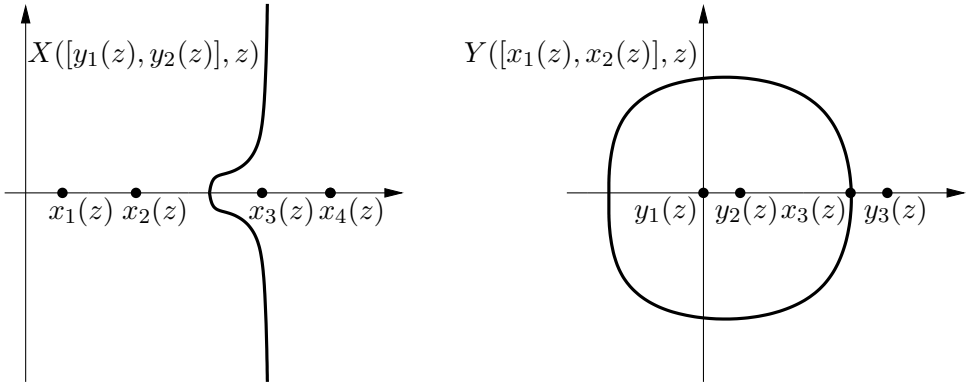}
\end{picture}
\end{center}
\vspace{-215mm}
\caption{The curves $X([y_{1}(z),y_{2}(z)],z)$ and $Y([x_{1}(z),x_{2}(z)],z)$}
\label{Locations_of_the_curves}
\end{figure}

The proofs of Lemmas~\ref{Properties_X_Y_0} and~\ref{Properties_curves_0} are
given in Part 5.3 of \cite{FIM} for $z=1/4$; they can be plainly extended up to
$z\in]0,1/4[$.

\medskip

We shall soon see that $zQ(x,0,z)$ and $z(y+1)Q(0,y,z)$
satisfy BVPs of Riemann-Carleman type. It
will turn out that the underlying boundary conditions for
$zQ(x,0,z)$ and $z(y+1)Q(0,y,z)$ hold \textit{formally} on the
curves $X([y_{1}(z),y_{2}(z)],z)$ and $Y([x_{1}(z),x_{2}(z)],z)$
respectively. However, these curves are \textit{not} included in the
unit disc, see Lemma~\ref{Properties_curves_0} above, therefore the
functions $zQ(x,0,z)$ and $z(y+1)Q(0,y,z)$ are \text{a priori} not
defined on them. For this reason we need, first of all, to continue the
generating functions up to these curves.
  This is why we state the following theorem; its proof is postponed to
   Section~\ref{CONTINUATION}.

\begin{thm}
\label{Thm_continuation} The functions
     $zQ(x,0,z)$ and $z(y+1)Q(0,y,z)$ can be continued as holomorphic functions
     from their unit disc up to $\mathbb{C}\setminus [x_{3}(z),x_{4}(z)]$ and
     $\mathbb{C}\setminus [y_{3}(z),y_{4}(z)]$ respectively. Furthermore, for
     any $y \in \mathbb{C}\setminus([y_{1}(z),y_{2}(z)]\cup [y_{3}(z),y_{4}(z)])$,
          \begin{equation}
          \label{ooo}
               zQ(X_{0}(y,z),0,z)+z(y+1)Q(0,y,z)-zQ(0,0,z)-X_{0}(y,z)y=0
          \end{equation}
     and for any $x\in \mathbb{C}\setminus([x_{1}(z),x_{2}(z)]\cup [x_{3}(z),x_{4}(z)])$,
          \begin{equation}
          \label{ooo1}
               zQ(x,0,z)+z(Y_{0}(x,z)+1)Q(0,Y_{0}(x,z),z) -zQ(0,0,z)-xY_{0}(x,z)=0.
          \end{equation}
\end{thm}
%\begin{rem}
%{
%For $y\in\{|y|\leq 1\}\setminus [y_{1}(z),y_{2}(z)]$ such that $|X_{0}(y,z)|\leq 1$, (\ref{ooo})
%immediately follows from (\ref{functional_equation}). Likewise, for $x\in\{|x|\leq 1\}\setminus
%[x_{1}(z),x_{2}(z)]$ such that $|Y_{0}(x,z)|\leq 1$, (\ref{ooo1}) is
%a straightforward consequence of (\ref{functional_equation}). The
%fact that Equations (\ref{ooo}) and (\ref{ooo1}) are more generally
%valid for $y\in\mathbb{C}\setminus\big([y_{1}(z),y_{2}(z)]\cup
%[y_{3}(z),y_{4}(z)]\big)$ and $x\in\mathbb{C}\setminus\big([y_{1}(z),y_{2}(z)]
%\cup[x_{3}(z),x_{4}(z)]\big)$ respectively will be shown in Section~\ref{CONTINUATION}.}
%\end{rem}

%\begin{rem}
%{
%
%In the proof of Theorem~\ref{Thm_continuation}, we will see that the
%function $Q(0,y,z)$ can also be holomorphically continued from the
%unit disc up to $\mathbb{C}\setminus [y_{3}(z),y_{4}(z)]$.}
%\end{rem}

It is immediate from Theorem~\ref{Thm_continuation} that for any
$z\in ]0,1/4[$, the function $Q(0,y,z)$ can be continued as a
holomorphic function from the unit disc up to $\mathbb{C}\setminus
[y_{3}(z),y_{4}(z)]$ as well: the point $y=-1$ cannot be a pole of
$Q(0,y,z)$ since the series $\sum_{j\geq 0,k\geq 0}q(0,j,k)
%y^j
z^{k}$
 %at $y=\pm 1$
 converges.

Now, we derive the above-mentioned boundary conditions satisfied by the functions
$zQ(x,0,z)$ and $z(y+1)Q(0,y,z)$ on the curves $X([y_{1}(z),y_{2}(z)],z)$ and
$Y([x_{1}(z),x_{2}(z)],z)$ respectively.

\begin{lem}
\label{RCproblem}
$zQ(x,0,z)$ and $z(y+1)Q(0,y,z)$ belong to the class of the functions~holomorphic in
$\mathscr{G}X([y_{1}(z),y_{2}(z)],z)$ and
$\mathscr{G}Y([x_{1}(z),x_{2}(z)],z)$ respectively, continuous up to
the boundary of the latter sets and satisfying the boundary conditions
\begin{equation}
     \label{BVP}
          \left.\begin{array}{ccccc}
          z\big[Q\big(t,0,z\big)\hspace{-0.5mm}-\hspace{-0.5mm}
          Q\big(\overline{t},0,z\big)\big]
          \hspace{-2.5mm}&=&\hspace{-2.5mm}
          tY_{0}\big(t,z\big)\hspace{-0.5mm}-\hspace{-0.5mm}\overline{t}Y_{0}\big(
          \overline{t},z\big), & &\hspace{-2.5mm}
          \forall t\in X([y_{1}(z),y_{2}(z)],z),\\
          z\big[\big(t+1\big)Q\big(0,t,z\big)\hspace{-0.5mm}-\hspace{-0.5mm}\big(\overline{t}+1\big)
          Q\big(0,\overline{t},z\big)\big]\hspace{-2.5mm}&=&\hspace{-2.5mm}X_{0}\big(t,z\big)t
          \hspace{-0.5mm}-\hspace{-0.5mm}X_{0}
          \big(\overline{t},z\big)\overline{t}, & &\hspace{-2.5mm}\forall t\in Y([x_{1}(z),x_{2}(z)],z).
          \end{array}\right.
     \end{equation}
\end{lem}

\begin{proof}
Due to Lemma~\ref{Properties_curves_0} and
Theorem~\ref{Thm_continuation}, it remains to prove
Equation~(\ref{BVP}) above. Let $y \in [y_{1}(z),y_{2}(z)]$ and let
$\hat{y}^{+}$ and $\hat{y}^{-}$ be close to $y$, such that
$\hat{y}^{+}$ is in the \textit{upper} half-plane and $\hat{y}^{-}$
in the \textit{lower} half-plane. Then we have (\ref{ooo}) for both
$\hat{y}^{+}$ and $\hat{y}^{-}$. If now $\hat{y}^{+}\to y$ and
$\hat{y}^{-}\to y$, then we obtain $X_{0}(\hat{y}^{+},z) \to
X_{0}^{+}(y,z)$ and $X_{0}(\hat{y}^{-},z) \to
X_{0}^{-}(y,z)=X_{1}^{+}(y,z)$. So we have proved that for any $y
\in [y_1(z),y_2(z)]$,
  \begin{eqnarray}
  \label{ooo2}
       \phantom{X_{1}}zQ(X_{0}^{+}(y,z),0,z)+z(y+1)Q(0,y,z)-zQ(0,0,z)
       \hspace{-3mm}&-&\hspace{-3mm}X_{0}^{+}(y,z)y=0,\phantom{.}\phantom{XX}\\
  \label{ooo3}
       \phantom{X_{0}}zQ(X_{1}^{+}(y,z),0,z)+z(y+1)Q(0,y,z)-zQ(0,0,z)
       \hspace{-3mm}&-&\hspace{-3mm}X_{1}^{+}(y,z)y=0.\phantom{,}\phantom{XX}
  \end{eqnarray}
Subtracting (\ref{ooo3}) from (\ref{ooo2}) gives that for $y \in [y_{1}(z), y_{2}(z)]$,
 \begin{equation*}
     \label{before_BVP}
          z\big[Q(X_{0}^{+}(y,z),0,z)-Q(X_{1}^{+}(y,z),0,z)\big]=
          X_{0}^{+}(y,z)y-X_{1}^{+}(y,z)y.
     \end{equation*}
Then, using the fact that for $k\in\{0,1\}$, $y\in[y_{1}(z),y_{2}(z)]$
and $z\in]0,1/4[$, $Y_{0}(X_{k}^{\pm}(y,z),z)=y$---which can be proved by
elementary considerations starting from Lemma~\ref{Properties_X_Y_0}---we get the first part of (\ref{BVP}). Likewise, we could prove
the second part of (\ref{BVP}).
\end{proof}

%Note that as a consequence of (\ref{ooo2}) and (\ref{ooo3}), Equation (\ref{ooo})
%is in some sense also satisfied for $y\in [y_{1}(z),y_{2}(z)]$; the same is true
%for Equation (\ref{ooo1}) and $x\in [x_{1}(z),x_{2}(z)]$.

\section{Results}
\label{Section_results}

Problems as in Lemma~\ref{RCproblem} are usually called BVPs of
Riemann-Carleman type, see e.g.\ Part~5.2.5 of \cite{FIM}. A
standard way to solve them consists in converting them into BVPs of
Riemann-Hilbert type (\text{i.e.}\ with boundary conditions on
segments) by using~{\it conformal gluing functions} (CGFs), as in
Equation (17.4') on page~$130$
%in Part 17.5
of \cite{GAK}.

\begin{defn}
\label{def_CGF}
     Let $\mathscr{C}\subset\mathbb{C}\cup \{\infty\}$ be an open and simply connected set,
     symmetrical w.r.t.\ the real axis and different from $\emptyset$, $\mathbb{C}$ and
     $\mathbb{C}\cup \{\infty\}$. A function $w$ is said to be a CGF for set $\mathscr{C}$
     if:
\begin{enumerate}

\item $w$ is meromorphic in $\mathscr{C}$;

\item $w$ establishes a conformal
     mapping of $\mathscr{C}$ onto the complex plane cut along some arc;

\item for all
     $t$ in the boundary of $\mathscr{C}$, $w\big(t\big)=w\big(\overline{t}\big)$.
\end{enumerate}
\end{defn}

Let $w(t,z)$ and $\widetilde{w}(t,z)$ be CGFs for
$\mathscr{G}X([y_{1}(z),y_{2}(z)],z)$ and
$\mathscr{G}Y([x_{1}(z),x_{2}(z)],z)$. The existence (but
\textit{no} explicit expression) of $w$ and $\widetilde{w}$ follows from
general results on conformal gluing, see e.g.\ Part 17.5 in \cite{GAK}.

Transforming the BVPs of Riemann-Carleman type into
BVPs of Riemann-Hilbert type thanks to $w$ and $\widetilde{w}$,
solving the latter and working out the solutions, we will prove the following.

\begin{thm}
\label{explicit_integral}
For $z\in]0,1/4[$ and $x\in\mathbb{C}\setminus [x_{3}(z),x_{4}(z)]$,
     \begin{eqnarray*}
          z\big[Q(x,0,z)-Q(0,0,z)\big]=\hspace{96mm}
          \\ xY_{0}(x,z)+\frac{1}{\pi}\int_{x_{1}(z)}^{x_{2}(z)}\frac{t\big[-d(t,z)\big]^{1/2}}
          {2a(t,z)}\left[\frac{\partial_{t}w(t,z)}{w(t,z)-w(x,z)}
          -\frac{\partial_{t}w(t,z)}{w(t,z)-w(0,z)}\right]\textnormal{d}t,
     \end{eqnarray*}
where $w$ is a CGF for the set $\mathscr{G}X([y_{1}(z),y_{2}(z)],z)$.

\medskip

For $z\in]0,1/4[$ and $y\in\mathbb{C}\setminus [y_{3}(z),y_{4}(z)]$,
     \begin{eqnarray*}
          z\big[(y+1)Q(0,y,z)-Q(0,0,z)\big]=\hspace{85mm}
          \\ X_{0}(y,z)y+\frac{1}{\pi}\int_{y_{1}(z)}^{y_{2}(z)}
          \frac{t\big[-\widetilde{d}(t,z)\big]^{1/2}}{2\widetilde{a}(t,z)}
          \left[\frac{\partial_{t}\widetilde{w}(t,z)}{\widetilde{w}(t,z)-
          \widetilde{w}(y,z)}-\frac{\partial_{t}\widetilde{w}(t,z)}
          {\widetilde{w}(t,z)-\widetilde{w}(0,z)}\right]
          \textnormal{d}t,
     \end{eqnarray*}
where $\widetilde{w}$ is a CGF for the set $\mathscr{G}Y([x_{1}(z),x_{2}(z)],z)$.

\medskip

For $z\in]0,1/4[$,
     \begin{equation*}
          Q(0,0,z)=-\frac{1}{\pi}\int_{y_{1}(z)}^{y_{2}(z)}\frac{t\big[
          -\widetilde{d}(t,z)\big]^{1/2}}{2\widetilde{a}(t,z)}\left[
          \frac{\partial_{t}\widetilde{w}(t,z)}{\widetilde{w}(t,z)-\widetilde{w}(-1,z)}
          -\frac{\partial_{t}\widetilde{w}(t,z)}{\widetilde{w}(t,z)-\widetilde{w}(0,z)}
          \right]\textnormal{d}t,
     \end{equation*}
where $\widetilde{w}$ is a CGF for the set $\mathscr{G}Y([x_{1}(z),x_{2}(z)],z)$.

\medskip

The function $Q(x,y,z)$ has then the explicit expression obtained by using the ones of
$Q(x,0,z)$, $Q(0,y,z)$ and $Q(0,0,z)$ in (\ref{functional_equation}).
\end{thm}

 % The function $Q_{i,j}(z)=\sum_{k \geq 0} q(i,j,k)z^k$ can be now
 % obtained from $Q(x,y,z)$ by Cauchy formulas for any real $z \in
 % ]0,1/4[$. Since the radius of convergence of the complex series
%$\sum_{k \geq 0} q(i,j,k)z^k$ is not smaller than $1/4$,
 % then $q_{i,j,k}(z)$ can be identified e.g.\ in terms  of
 %  derivatives as $\lim_{z \to 0+} Q_{i,j}^{(k)}(z)/k!$.

All functions in the integrands above
%in Theorem~\ref{explicit_integral}
are explicit, except for the CGFs $w$
and $\widetilde{w}$. In \cite{FIM}, suitable CGFs are computed
\textit{implicitly} by means of the inverse of some known function,
see Equations (\ref{expression_CGFx}) and (\ref{expression_CGFy}) in
Section~\ref{Study_CGF} for the details. Starting from this
representation, we are able to make \textit{explicit} these
functions for the case of Gessel's walks.
In order to state the result, we need to define
     \begin{equation}
     \label{GGG}
          \begin{array}{ccl}
          \displaystyle G_{2}(z)&=&(4/27)\big(1+224z^{2}+256z^{4}\big),\\
          \displaystyle G_{3}(z)&=&(8/729)\big(1+16z^{2}\big)\big(1-24z+
          16z^{2}\big)\big(1+24z+16z^{2}\big).
          \end{array}
     \end{equation}
We define also $K(z)$ as the unique real positive solution to
     \begin{equation}
     \label{KKK}
          K^{4}-G_{2}(z)K^{2}/2-G_{3}(z)K-G_{2}(z)^{2}/48=0.
     \end{equation}
With the notations $r_{k}(z)=[G_{2}(z)-\exp(2k\imath
\pi/3)(G_{2}(z)^{3}-27G_{3}(z)^{2})^{1/3}]/3$ for $k\in \{0,1,2\}$, we have
$K(z)=[-r_{0}(z)^{1/2}+r_{1}(z)^{1/2}+r_{2}(z)^{1/2}]/2$. We finally
define
     \begin{equation}
     \label{def_F_F_tilde}
          \left.\begin{array}{ccccc}
          F(t,z)&=&\displaystyle\frac{1-24z+16z^{2}}{3}&-&
          \displaystyle\frac{4(1-4z)^{2}}{z}
          \frac{t^{2}}{(t-x_{2}(z))(t-1)^{2}(t-x_{3}(z))},\\
          \widetilde{F}(t,z)&=&\displaystyle\frac{1-24z+16z^{2}}{3}
          &+&\displaystyle\frac{4(1-4z)^{2}}{z}
          \frac{t(t+1)^{2}}{\hspace{4.5mm}[(t-x_{2}(z))(t-x_{3}(z))]^{2}\hspace{4.5mm}}.
     \end{array}\right.
     \end{equation}

\begin{thm}
\label{w_w_tilde_algebraic}
A suitable CGF for the set $\mathscr{G}X([y_{1}(z),y_{2}(z)],z)$
 is the unique function having a
pole at $x_{2}(z)$ and solution to
       \begin{equation}
       \label{eq_w_w_tilde_algebraic}
       \hspace{-1mm}\left.\begin{array}{ccc}
            w^{3}-w^{2}\big[F(t,z)\hspace{-3mm}&+&\hspace{-25mm}2K(z)\big]
            +w\big[2K(z)F(t,z)+K(z)^{2}/3+G_{2}(z)/2\big]\\
            \hspace{-2mm}&-&\hspace{-2mm}\big[K(z)^{2}F(t,z)+19G_{2}(z)K(z)/18+
            G_{3}(z)-46K(z)^{3}/27\big]=0.
       \end{array}\right.
       \end{equation}
Likewise, a suitable CGF for the set
$\mathscr{G}Y([x_{1}(z),x_{2}(z)],z)$ is the unique function having
a pole at $x_{3}(z)$ and solution to the equation obtained from (\ref{eq_w_w_tilde_algebraic})
by replacing $F$ by $\widetilde{F}$, see~(\ref{def_F_F_tilde}).
\end{thm}

%\begin{cor}
%\label{explicit_numbers}
%For any $(i,j)\in\mathbb{Z}_{+}^{2}$ and $k\geq 0$, the following number is explicit:
%     \begin{equation*}
%          \#\big\{\text{Gessel's walks starting at}\ (0,0)\ \text{and
%          ending at}\ (i,j)\ \text{in time}\ k\big\}.
%     \end{equation*}
%\end{cor}

  The function $Q(x,y,z)$
   being found in the domain $\{|x|<1, |y|<1, z\in ]0,1/4[\}$
 thanks to Theorems~\ref{explicit_integral} and~\ref{w_w_tilde_algebraic},
  its coefficients $Q_{i,j}(z)=\sum_{k \geq 0} q(i,j,k)z^k$ can be
 obtained most easily by the use of Cauchy formulas, for any real $z \in
  ]0,1/4[$. Since the radius of convergence of the complex series
$\sum_{k \geq 0} q(i,j,k)z^k$ is not smaller than $1/4$, the
  numbers $q(i,j,k)$ of Gessel's walks  can then  be identified
  e.g.\ in terms  of the limits
    of the successive derivatives as $z \to 0+$ (i.e.\ as $z>0$ goes to $0$):
    \begin{equation*}
         q(i,j,k)=\lim_{z \to 0+}
         \frac{1}{k!}
         \frac{\text{d}^{k}Q_{i,j}(z)}{\text{d} z^{k}}.
    \end{equation*}

    %\medskip

Let us now outline two facts about the expressions of $Q(x,0,z)$,
$Q(0,y,z)$ and~$Q(0,0,z)$ stated in Theorems~\ref{explicit_integral}
and~\ref{w_w_tilde_algebraic}.

\begin{rem}
\label{sdf}{\rm The fact that  (\ref{functional_equation}) is valid
at least on $\{|x|\leq 1, |y|\leq 1, |z|<1/4\}$ gives that for any triplet
$(\hat{x}, \hat{y},z)\in\{|x|\leq 1, |y|\leq 1, |z|<1/4\}$ such that
$L(\hat{x},\hat{y},z)=0$, the right-hand side of
(\ref{functional_equation}) equals zero, in such a way that
     \begin{equation}
     \label{mainzz}
          z\big[Q(\hat{x},0,z)-Q(0,0,z)\big]+z\big[(\hat{y}+1)Q(0,\hat{y},z)-
          Q(0,0,z)\big]+zQ(0,0,z)-\hat{x}\hat{y}=0.
     \end{equation}
We deduce that
     \begin{equation}
     \label{sss1}
          zQ(0,0,z)= -z\big[Q(\hat{x},0,z)-Q(0,0,z)\big]-z\big[(\hat{y}+1)
          Q(0,\hat{y},z)-Q(0,0,z)\big]+\hat{x}\hat{y},
     \end{equation}
where the functions in square brackets in the right-hand side of (\ref{sss1}) are given
thanks to Theorem~\ref{explicit_integral}.

To get the expression of $zQ(0,0,z)$ given in Theorem~\ref{explicit_integral},
we have chosen to substitute
$(\hat{x},\hat{y},z)=(0,-1,z)$ in (\ref{sss1}), which is suitable, since
with Lemma~\ref{Properties_X_Y_0} we have $X_{0}(-1,z)=0$.

Moreover, a consequence of Theorem~\ref{Thm_continuation} is that
(\ref{mainzz}) is valid not only on
$\{L(x,y,z)=0\}\cap \{|x|\leq 1, |y|\leq 1,z\in]0,1/4[\}$ but in a much larger
domain of the algebraic curve $\{L(x,y,z)=0\}$. Namely, if
$(\hat{x},\hat{y},z)$ is such that $z\in]0,1/4[$ and
$\hat{y}=Y_{0}(\hat{x},z)$ or $\hat{x}=$ $X_{0}(\hat{y},z)$, then
(\ref{mainzz}) is still valid. Substituting any triplet
$(\hat{x},\hat{y},z)$ lying in this domain into (\ref{mainzz}) yields
$zQ(0,0,z)$ as in (\ref{sss1}).}
\end{rem}

\begin{rem}
{\rm In Theorem~\ref{explicit_integral}, $z[Q(x,0,z)-Q(0,0,z)]$ and
$z[(y+1)Q(0,y,z)-Q(0,0,z)]$ are written as the sums of two functions
not holomorphic but singular near $[x_{1}(z),x_{2}(z)]$ and
$[y_{1}(z),y_{2}(z)]$ respectively. The sums of these two singular
functions are of course holomorphic near these segments, since the latter
are included in the unit disc, according to Section~\ref{reduc}. By an application of the residue
theorem as in Section~4 of \cite{KR}, we could write both generating
functions as functions manifestly holomorphic near these segments
and having in fact their singularities near $[x_{3}(z),x_{4}(z)]$
and $[y_{3}(z),y_{4}(z)]$ respectively.}
\end{rem}
 We conclude the discussion of Theorems~\ref{Thm_continuation},
 \ref{explicit_integral} and~\ref{w_w_tilde_algebraic} with the following remark.

\begin{rem}
{\rm
With the analytical approach proposed in this article, it would
be possible, without additional difficulty, to obtain explicitly the
generating function of the number of walks of length $k$, starting at an arbitrary
initial state $(i_{0},j_{0})$ and ending at $(i,j)$. Indeed,
the only difference is that the product $x y$ in (\ref{functional_equation})
would then be replaced by $x^{i_{0}+1}y^{j_{0}+1}$.
}
\end{rem}

%\begin{rem}
%Making in Theorem~\ref{explicit_integral} the changes of variables $w=w(t,z)$
%and $\widetilde{w}=\widetilde{w}(t,z)$, we obtain that the generating functions
%$zQ(x,0,z)$ and $z(y+1)Q(0,y,z)$ are essentially Cauchy-type integrals of algebraic
%functions.
%
%In particular, it could be deduced from the work \cite{PAK}---which gives
%criteria for a Cauchy-type integral of an algebraic function to be algebraic---
%that as functions of $x$ and $y$ respectively, $zQ(x,0,z)$ and $z(y+1)Q(0,y,z)$
%are algebraic functions, what would give another proof to some results
%contained in \cite{BK}.
%\end{rem}

The rest of the article is organized as follows. In
Section~\ref{SectionTheorem1}, we prove
Theorem~\ref{explicit_integral}. In Section~\ref{Study_CGF}, we give
the proof of Theorem~\ref{w_w_tilde_algebraic}. There, the general
implicit representation of the CGFs inspired by \cite{FIM} is given
and developed for the case of Gessel's walks. The proof of
Theorem~\ref{Thm_continuation} is postponed to the last
Section~\ref{CONTINUATION}.

\section{Proof of Theorem~6}
\label{SectionTheorem1}

% and Corollary~\ref{explicit_numbers}.
%
%In Section~\ref{Study_CGF} we prove Theorem~\ref{w_w_tilde_algebraic}. There
%the implicit representation of the CGF given in \cite{FIM} (and recalled here in
%Subsections~\ref{Subsection_Uniformization} and~\ref{Implicit_FIM}) in a general
%setting is developed in Subsection~\ref{Proofss} for Gessel's walks.
%
%The proof of Theorem~\ref{Thm_continuation} is postponed to the last
%Section~\ref{CONTINUATION}. The main idea of the holomorphic
%continuation procedure is borrowed again from \cite{FIM}, we show
%how it works with the parameter $z \in]0,1/4[$.
%
%Finally, we give the proof of Theorem~\ref{explicit_integral}.

%\begin{proof}[Proof of Theorem~\ref{explicit_integral}]
The proof is composed of three steps.
\medskip

\noindent{\it Step 1.} We solve the BVPs of
Riemann-Carleman type with conditions (\ref{BVP}) by transforming
them into BVPs of Riemann-Hilbert type as in
Parts 5.2.3--5.2.5 of \cite{FIM} or in Part 17(.5) of \cite{GAK}. The only
notable difference from \cite{FIM} is that the index of our problems
is zero; this is why the solution $zQ(x,0,z)$ is found in
$\mathscr{G}X([y_{1}(z),y_{2}(z)],z)$ up to an additive function of
$z$, as
     \begin{equation}
     \label{intint}
          zQ(x,0,z)=\frac{1}{2\pi\imath}\int_{X([y_{1}(z),y_{2}(z)],z)} tY_{0}(t,z)
          \frac{\partial_{t}w(t,z)}{w(t,z)-w(x,z)}\text{d}t,
          \ \ \forall x \in \mathscr{G}X([y_{1}(z),y_{2}(z)],z)
     \end{equation}
where $w$ is the CGF used for $\mathscr{G}X([y_{1}(z),y_{2}(z)],z)$.
Similarly, we could write an integral representation for
$z(y+1)Q(0,y,z)$, up to some additive function of $z$.

\medskip

\noindent{\it Step 2.}
  We transform these representations into
  the integrals on real segments written in the statement of
Theorem~\ref{explicit_integral}. This step is inspired by \cite{KR}.

  Let $C(\epsilon,z)$ be any contour such that:
     \begin{enumerate}
          \item \label{property_connected}
                $C(\epsilon,z)$ is connected and contains $\infty$;
          \item \label{property_included}
                $C(\epsilon,z)\subset  (\mathscr{G}X([y_{1}(z),y_{2}(z)],z)
                \cup X([y_{1}(z),y_{2}(z)],z))\setminus [x_{1}(z),x_{2}(z)]$;
          \item \label{property_limit}
                $\lim_{\epsilon\to 0}C(\epsilon,z)=X([y_{1}(z),y_{2}(z)],z)
                \cup S(z)$, with $S(z)$ the segment
                $[x_{1}(z),X(y_{2}(z),z)]$
                traversed from $X(y_{2}(z),z)$ to $x_{1}(z)$ along the lower
                edge of the slit and then back to $X(y_{2}(z),z)$ along the
                upper edge.
     \end{enumerate}
\begin{figure}[!ht]
\begin{center}
\begin{picture}(000.00,740.00)
\hspace{-112mm}\includegraphics{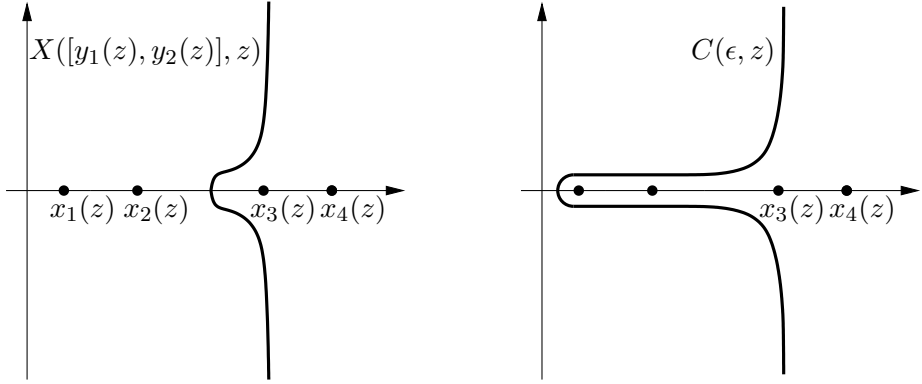}
\end{picture}
\end{center}
\vspace{-213mm}
\caption{The curve $X([y_{1}(z),y_{2}(z)],z)$ and the new contour of integration $C(\epsilon, z)$}
\label{New_contour}
\end{figure}
Let $\mathscr{G}C(\epsilon,z)$ be the connected component of $\mathbb{C}\setminus C(\epsilon,z)$
which does not contain $x_{3}(z)$. Now we apply the residue theorem
to the integrand of (\ref{intint}) on the contour $C(\epsilon,z)$.
Thanks to Lemma~\ref{Properties_X_Y_0} and Property
\ref{property_included} above, $t\mapsto tY_{0}(t,z)$ is holomorphic
in $\mathscr{G}C(\epsilon,z)$. Likewise, by using Definition
\ref{def_CGF} and Property~\ref{property_included}, we reach the conclusion that
$\partial_{t}w(t,z)/[w(t,z)-w(x,z)]$ is meromorphic on
$\mathscr{G}C(\epsilon,z)$, with a single pole at $t=x$. For this
reason,
     \begin{equation}
     \label{application_residue}
          \frac{1}{2\pi\imath}\int_{C(\epsilon,z)} tY_{0}(t,z)
          \frac{\partial_{t}w(t,z)}{w(t,z)-w(x,z)}\text{d}t=xY_{0}(x,z).
     \end{equation}
Then, letting $\epsilon$ tend to $0$, using Equations
(\ref{intint})--(\ref{application_residue}) and Property
\ref{property_limit} of the contour, we derive that, up to an
additive function of $z$,
     \begin{equation}
     \label{S_present}
          zQ(x,0,z)=xY_{0}(x,z)-\frac{1}{2\pi\imath}\int_{S(z)} tY_{0}(t,z)
          \frac{\partial_{t}w(t,z)}{w(t,z)-w(x,z)}\text{d}t.
     \end{equation}
Since for any $x\in\mathscr{G}X([y_{1}(z),y_{2}(z)],z)$, the integrand
of (\ref{S_present}) is,
as a function of the variable $t$,
holomorphic in $]x_{2}(z),X(y_{2}(z),z)[$, we have
     \begin{equation*}
          \int_{S(z)} tY_{0}(t,z)\frac{\partial_{t}w(t,z)}{w(t,z)-w(x,z)}\text{d}t
          =\int_{x_{1}(z)}^{x_{2}(z)}
          \big[tY_{0}^{+}(t,z)-tY_{0}^{-}(t,z)\big]
          \frac{\partial_{t}w(t,z)}{w(t,z)-w(x,z)}\text{d}t.
     \end{equation*}
Using (\ref{values_cuts}), we then immediately obtain the expression
of $z[Q(x,0,z)-Q(0,0,z)]$ stated in Theorem~\ref{explicit_integral}
for $x\in\mathscr{G}X([y_{1}(z),y_{2}(z)],z)$. Likewise, we could
obtain the expression of $z[(y+1)Q(0,y,z)-Q(0,0,z)]$ written in
Theorem~\ref{explicit_integral} for
$y\in\mathscr{G}Y([x_{1}(z),x_{2}(z)],z)$. The formula for
$Q(0,0,z)$ has already been proved in Remark~\ref{sdf}.

\medskip

\noindent{\it Step 3.} In order to complete the proof of
Theorem~\ref{explicit_integral}, we have to show that the integral
representations of $z[Q(x,0,z)-Q(0,0,z)]$ and
$z[(y+1)Q(0,y,z)-Q(0,0,z)]$ hold not only on
$\mathscr{G}X([y_{1}(z),y_{2}(z)],z)$ and
$\mathscr{G}Y([x_{1}(z),x_{2}(z)],z)$ but on the domains
$\mathbb{C}\setminus [x_{3}(z),x_{4}(z)]$ and $\mathbb{C}\setminus
[y_{3}(z),y_{4}(z)]$ respectively.

It is clear that they can be
  continued  up to
  $\mathbb{C}\setminus  ([x_{3}(z),x_{4}(z)]\cup (w^{-1}(w([x_{1}(z),
  x_{2}(z)],$ $z))\setminus [x_{1}(z),
  x_{2}(z)],z))$ and $\mathbb{C}\setminus ([y_{3}(z),y_{4}(z)]
  \cup(\widetilde{w}^{-1}(\widetilde{w}
  ([y_{1}(z),y_{2}(z)],z))\setminus [y_{1}(z),y_{2}(z)]))$ respectively.
  To conclude the proof of Theorem~\ref{explicit_integral},
 it therefore suffices to show that
    $w^{-1}(w([x_{1}(z),x_{2}(z)],z))
  \setminus [x_{1}(z),x_{2}(z)]=\emptyset$ and
   $\widetilde{w}^{-1}(\widetilde{w}([y_{1}(z),y_{2}(z)],z))
  \setminus [y_{1}(z),y_{2}(z)]=\emptyset$. This is the subject of
Proposition~\ref{cc}; it is postponed to Section~\ref{Study_CGF}
because all necessary facts about the functions $w$ and $\widetilde{w}$ are
proved there.
$\hfill \square$

\section{Study of the conformal gluing functions}
\label{Study_CGF}

\noindent{\bf Notation.} To be concise we drop,
from now on, the dependence on $z$ of all quantities.

\medskip

The main subject of Section~\ref{Study_CGF} is to prove Theorem~\ref{w_w_tilde_algebraic}. For
this we shall define two functions, namely $w$ in (\ref{expression_CGFx}) and $\widetilde{w}$
in (\ref{expression_CGFy}), which thanks to Part~5.5 of \cite{FIM} are known to be suitable
CGFs for the sets $\mathscr{G}X([y_{1},y_{2}])$ and $\mathscr{G}Y([x_{1},x_{2}])$ respectively.
We shall then show that these functions satisfy the conclusions of Theorem~\ref{w_w_tilde_algebraic}.

%These definitions of the CGFs given in \cite{FIM} are recalled here in Subsection~\ref{Implicit_FIM},
%see particularly (\ref{expression_CGFx}) and (\ref{expression_CGFy}). They require to define some functions
%on a uniformization of the algebraic curve $\{(x,y)\in \mathbb{C}^{2}: L(x,y,z)=0\}$,
%so that we begin Section~\ref{Study_CGF} by studying a suitable uniformization of this curve
%---note that this Subsection~\ref{Subsection_Uniformization} is also necessary in
%Section~\ref{CONTINUATION}, where we will prove Theorem~\ref{Thm_continuation}.

\subsection{Uniformization}
\label{Subsection_Uniformization}

\medskip

Let us begin Section~\ref{Study_CGF} with studying a {\it uniformization} of the algebraic curve
$\mathscr{L}=\{(x,y)\in (\mathbb{C}\cup\{\infty\})^{2}: L(x,y,z)=0\}$,
where $L(x,y,z)=
x y z \big[1/x+1/(x y)+x + x y -1/z\big]
$
is the kernel appearing in (\ref{functional_equation}).

     \begin{prop}
     \label{genus_Riemann_surface}
          For any $z\in]0,1/4[$, $\mathscr{L}$ is a Riemann surface of genus one.
     \end{prop}
\begin{proof}
We have shown in Section~\ref{reduc} that $L(x,y,z)=0$
if and only if $[b(x)+2a(x)y]^{2}=d(x)$. The Riemann surface of
the square root of a polynomial which has four distinct roots of order
one having genus one (see e.g.\ Part 4.9 in \cite{JS}, particularly pages 162--163),
%Part 4.9 in \cite{JS} corresponds to 157--163
the genus
of $\mathscr{L}$ is also one.
\end{proof}

With Proposition~\ref{genus_Riemann_surface}, it is immediate that
$\mathscr{L}$ is isomorphic to some torus. In other words, there
exists a two-dimensional lattice $\Omega$ such that $\mathscr{L}$
is isomorphic to $\mathbb{C}/\Omega$. Such a suitable lattice
$\Omega$ (in fact the \textit{only possible} lattice, up to
a homothetic transformation)
is found explicitly in Parts 3.1 and 3.3 of \cite{FIM}, namely
$\omega_{1}\mathbb{Z}+\omega_{2}\mathbb{Z}$, where
     \begin{equation}
     \label{def_omega_1_2}
          \omega_{1}= \imath\int_{x_{1}}^{x_{2}}
          \frac{\text{d}x}{[-d(x)]^{1/2}},
          \ \ \ \ \ \ \ \ \ \ \omega_{2}= \int_{x_{2}}^{x_{3}}
          \frac{\text{d}x}{[d(x)]^{1/2}}.
     \end{equation}

We shall now give a uniformization of surface $\mathscr{L}$,
\text{i.e.}\ we shall make explicit $x(\omega)$ and $y(\omega)$,
two functions elliptic w.r.t.\ lattice $\Omega$ and such that
$\mathscr{L}=\{(x(\omega),y(\omega)), \omega\in
\mathbb{C}/\Omega\}$. By using the same arguments as in Part 3.3 of
\cite{FIM}, we immediately see that we can take
%     \begin{equation}
%     \label{uniformization}
%          \left\{\begin{array}{ccc}
%          x(\omega)&=&\displaystyle \phantom{\frac{1}{\frac{1}{\frac{1}{1}}}}
%          x_{4}+\frac{d'(x_{4})}
%          {\wp(\omega)-d''(x_{4})/6},
%          \phantom{\frac{1}{\frac{1}{\frac{1}{1}}}}\\
%          y(\omega)&=&\displaystyle \frac{1}{2a(x(\omega))}\left[-b(x(\omega))+\frac{
%          d'(x_{4})\wp'(\omega)}{2(\wp(\omega)-
%          \partial_{1}^{2}d''(x_{4})/6)^{2}}\right].
%          \end{array}\right.
%     \end{equation}
     \begin{equation}
     \label{uniformization}
          x(\omega)=x_{4}+\frac{d'(x_{4})}
          {\wp(\omega)-d''(x_{4})/6},
          \ \ y(\omega)=\displaystyle \frac{1}{2a(x(\omega))}
          \left[-b(x(\omega))+\frac{d'(x_{4})
          \wp'(\omega)}{2[\wp(\omega)-
          d''(x_{4})/6]^{2}}\right],
     \end{equation}
where $\wp$ is the Weierstrass elliptic function with periods
$\omega_{1},\omega_{2}$.

%By convenience, we will consider, from now on, that the coordinates of the uniformization
%$x$ and $y$ are defined on $\mathbb{C}/\Omega$ rather than on $\mathbb{C}$.

It is well-known (see e.g.\ (6.7.26) on page $159$ in \cite{LAW}) that $\wp$ is characterized
by its invariants $g_{2},g_{3}$ through
     \begin{equation}
     \label{relation_wp_wp'}
          \wp'(\omega)^{2}=4\wp(\omega)^{3}-g_{2}\wp(\omega)-g_{3}.
     \end{equation}
     \begin{lem}
     \label{lemma_g_2_g_3}
     Invariants $g_{2},g_{3}$ of $\wp$ are equal to:
     \begin{equation}
     \label{g_2_g_3}
          g_{2}=(4/3)\big(1-16z^{2}+16z^{4}\big),
          \ \ \ \ \ g_{3}=-(8/27)\big(1-8z^{2}\big)
          \big(1-16z^{2}-8z^{4}\big).
     \end{equation}
     \end{lem}

\begin{proof}
We have that $4\wp(\omega)^{3}-g_{2}\wp(\omega)-g_{3}=
4[\wp(\omega)-\wp(\omega_{1}/2)][\wp(\omega)-\wp
([\omega_{1}+\omega_{2}]/2)][\wp(\omega)-\wp(\omega_{2}/2)]$, see
e.g.\ (6.7.16) on page $158$ and (6.7.26) on page $159$ in \cite{LAW}.
In particular, invariants $g_{2},g_{3}$ can be
expressed in terms of the values of $\wp$ at the half-periods. But
it is clear by construction---and proved in Part 3.3 of \cite{FIM}---that setting
     \begin{equation}
     \label{ffu}
          f(t)=\frac{d''(x_{4})}{6}+\frac{d'(x_{4})}{t-x_{4}},
     \end{equation}
we have $\wp(\omega_{1}/2) =f(x_{3})$,
$\wp([\omega_{1}+\omega_{2}]/2)=f(x_{2})$ and
$\wp(\omega_{2}/2)=f(x_{1})$. Lemma~\ref{lemma_g_2_g_3} follows then
from a direct calculation.
\end{proof}

%With Equations (\ref{def_omega_1_2}), (\ref{uniformization})
%and Lemma~\ref{lemma_g_2_g_3}, the uniformization $(x(\omega),y(\omega))$ is completely
%explicit.

For an upcoming use, we would like now to know the inverse images through
the uniformization of the important cycles that are the branch cuts, that is to say
$x^{-1}([x_{1},x_{2}])$, $x^{-1}([x_{3},x_{4}])$, $y^{-1}([y_{1},y_{2}])$
and $y^{-1}([y_{3},y_{4}])$. In this perspective, we introduce a new period, namely
     \begin{equation}
     \label{def_omega_3}
          \omega_{3}= \int_{-\infty}^{x_{1}}
          \frac{\text{d}x}{[d(x)]^{1/2}}.
     \end{equation}
We will extensively use that $\omega_{3}\in]0,\omega_{2}[$---this fact is proved in Lemma~3.3.3 on page 47 of \cite{FIM}.
     \begin{prop}
     \label{transformation_cycles_uniformization}
          We have $x^{-1}([x_{1},x_{2}])=[0,\omega_{1}[+\omega_{2}/2$ as well as
          $x^{-1}([x_{3},x_{4}])=[0,\omega_{1}[$, $y^{-1}([y_{1},y_{2}])
          =[0,\omega_{1}[+[\omega_{2}+\omega_{3}]/2$ and $y^{-1}([y_{3},y_{4}])=
          [0,\omega_{1}[+\omega_{3}/2$.
     \end{prop}
Proposition~\ref{transformation_cycles_uniformization} follows from repeating
the arguments given in Part~3.3 of \cite{FIM}. It is illustrated in Figure
\ref{Locations_of_the_cycles}.

\medskip

Let $S(x,y)=1/x+1/(x y)+x + xy$ be the jump generating function
of Gessel's walks. Consider the two birational transformations
of $(\mathbb{C}\cup\{\infty\})^{2}$
     \begin{equation*}
     \label{expression_psi_phi}
          \Psi(x,y)=\left(x,\frac{1}{x^{2}y}\right),
          \ \ \ \ \ \Phi(x,y)=\left(\frac{1}{x y},y\right).
     \end{equation*}

They satisfy $\Psi^{2}=\Phi^{2}=\text{id}$ and $S\circ \Psi=S\circ \Phi=S$.
Then, as in Part 2.4 of \cite{FIM}, we define the \textit{group of the walk}
as group $G$ generated by $\Psi$ and $\Phi$. It is shown in
\cite{BMM} that for Gessel's walks, $G$ is of order eight: in other
words, $\inf \{n>0: (\Phi\circ\Psi)^{n}=\text{id}\}=4$.

If $(x,y)\in(\mathbb{C}\cup\{\infty\})^{2}$ is such that $L(x,y,z)=0$ and if $\theta$ is any
element of $G$, then clearly $L(\theta(x,y),z)=0$. This implies that group
$G$ can also be understood as a group of automorphisms of the algebraic curve
$\mathscr{L}$.
It is then shown in (3.1.6) and (3.1.8) in Part 3.1 of \cite{FIM} that the automorphisms $\Psi$ and $\Phi$
of $\mathscr{L}$ become the automorphisms of $\mathbb{C}/\Omega$
     \begin{equation}
     \label{expression_automorphisms_torus}
          \psi(\omega)=-\omega,
          \ \ \ \ \ \phi(\omega)=-\omega+\omega_{3}
     \end{equation}
respectively. They are such that $\psi^{2}=\phi^{2}=\text{id}$,
$x\circ\psi=x$, $y\circ\psi=1/(x^{2}y)$, $x\circ\phi=1/(x y)$ and
$y\circ\phi=y$.

A crucial fact is the following.
     \begin{prop}
     \label{www}
     \label{omega_3=(3/4)omega_2}
          For all $z\in]0,1/4[$, we have $\omega_{3}=3\omega_{2}/4$.
     \end{prop}
\begin{proof}
Since the group generated by $\Psi$ and $\Phi$ is of order eight,
the group generated by
$\psi$ and $\phi$ is also of order eight, for any $z\in]0,1/4[$,
see Section~3 of \cite{groupinfinite}.
In other words, $\inf \{n>0: (\phi\circ\psi)^{n}
=\text{id}\}=4$. With (\ref{expression_automorphisms_torus}),
this immediately implies that $4\omega_{3}$
is some point of the lattice $\Omega$, contrary to $\omega_{3}$, $2\omega_{3}$
and $3\omega_{3}$. But we already know that
$\omega_{3}\in]0,\omega_{2}[$, so that two possibilities remain:
either $\omega_{3}=\omega_{2}/4$ or $\omega_{3}=3\omega_{2}/4$.

In addition, essentially because the covariance of Gessel's walks is
positive, we can use the same arguments as in Section~4 of \cite{KR}
and this way, we conclude that $\omega_{3}$ has to be larger than
$\omega_{2}/2$, which finally yields Proposition~\ref{omega_3=(3/4)omega_2}.
%By the same argument as in Section~4 of \cite{KR}, this is equivalent to the fact
%that the quantity
%     \begin{equation*}
%          \left|\begin{array}{ccc}
%          1&\phantom{-}1&0\\
%          0&-z&0\\
%          0&\phantom{-}1&1
%          \end{array}\right|
%     \end{equation*}
%is negative, which is manifestly the case.
\end{proof}

%Let us pursue Section~\ref{Study_CGF} by the study of the implicit expression and the global properties of the CGFs.

\subsection{Global properties of CGFs}
\label{Implicit_FIM}

As said in Section~\ref{Section_results}, the \textit{existence} of CGFs for the sets
$\mathscr{G}X([y_{1},y_{2}])$ and $\mathscr{G}Y([x_{1},x_{2}])$ follows from general
results on conformal gluing, see e.g.\ page 130 of Part 17.5 in \cite{GAK}. Finding
\textit{explicit} expressions for CGFs is far more problematic: indeed, except for a
few particular cases, like e.g.\ discs or ellipses, obtaining the expression of a CGF
for a given set is, in general, quite a difficult task.

But by using the same analysis as in Part 5.5 of \cite{FIM}, we obtain explicitly
suitable CGFs for $\mathscr{G}X([y_{1},y_{2}])$ and $\mathscr{G}Y([x_{1},x_{2}])$.
Before writing the expression of these CGFs, we
recall that $\wp$, and thus also $x$ with (\ref{uniformization}), take each value of $\mathbb{C}\cup\{\infty\}$
twice on $[0,\omega_{2}[\times [0,\omega_{1}/\imath[$, but are one-to-one from
the half-parallelogram $(]0,\omega_{2}/2[\times [0,\omega_{1}/\imath[)\cup [0,\omega_{1}/2] \cup ([0,\omega_{1}/2]+\omega_{2}/2)$
onto $\mathbb{C}\cup \{\infty\}$---indeed, see Corollary~3.10.8 in \cite{JS} and remember that $\wp$ is even. In particular, on the latter
domain, $x$ admits an inverse function, that we denote by $x^{-1}$.

Then, with Part~5.5.2.1 of \cite{FIM}, we state:
     \begin{equation}
     \label{expression_CGFx}
          w(t)=\wp_{1,3} \big(x^{-1}(t)-[\omega_{1}+\omega_{2}]/2\big),
     \end{equation}
where $\wp_{1,3}$ is the Weierstrass elliptic function with periods
$\omega_{1},\omega_{3}$,
 $x^{-1}$ the inverse function
of the first coordinate of the uniformization
(\ref{uniformization}) and where
$\omega_{1},\omega_{2},\omega_{3}$ are defined in
(\ref{def_omega_1_2}) and (\ref{def_omega_3}).
   With (\ref{uniformization}), we note that
  \begin{equation}
  \label{zgh}
       x^{-1}(t)=\wp^{-1} (f(t)),
  \end{equation}
where $f$ is defined in (\ref{ffu}).

In Section~4 of \cite{KR}, we have studied some properties of
the function $w$ defined in (\ref{expression_CGFx}) and we have
shown that if $\omega_{3}>\omega_{2}/2$ (which is actually the case
here, see Proposition~\ref{omega_3=(3/4)omega_2}), then
$w$ is meromorphic on $\mathbb{C}\setminus [x_{3},x_{4}]$ and has
there a single pole, which is at $x_{2}$.

Let us now notice that
%We could write a formula similar to (\ref{expression_CGFx}).
     \begin{equation}
     \label{expression_CGFy}
          \widetilde{w}(t)=w(X_{0}(t))
     \end{equation}
is a suitable CGF for the set $\mathscr{G}Y([x_{1},x_{2}])$. Indeed,
as we verify below, the three items of Definition~\ref{def_CGF} are
satisfied.

Firstly, we immediately deduce from Lemma~\ref{me_fu} and from the inclusion
$\mathscr{G}Y([x_{1},x_{2}])\subset\mathbb{C}\setminus [y_{3},y_{4}]$ stated in
\ref{incc} of Lemma~\ref{Properties_curves_0} that
$\widetilde{w}$ defined in (\ref{expression_CGFy}) is meromorphic on
$\mathscr{G}Y([x_{1},x_{2}])$.

Secondly, using that $w$ is a CGF for $\mathscr{G}X([y_{1},y_{2}])$ as well as
Property~\ref{fi_fp} of Lemma~\ref{further_properties}, we reach the conclusion that
$\widetilde{w}$ establishes a conformal mapping of $\mathscr{G}Y([x_{1},x_{2}])$ onto
the complex plane cut along some arc.

Thirdly, once again with Property~\ref {fi_fp} of Lemma~\ref{further_properties}, we get
$X_{0}(Y_{0}(x))=X_{0}(Y_{1}(x))=x$ for $x\in[x_{1},x_{2}]$. As an immediate consequence,
for $t\in Y([x_{1},x_{2}])$ we have $\widetilde{w}\big(t\big)=\widetilde{w}\big(\overline{t}\big)$.

   \begin{lem}
   \label{further_properties}
   The two following properties hold.
   \begin{enumerate}
        \item \label{fi_fp} $X_{0}:\mathscr{G}Y([x_{1},x_{2}])\setminus [y_{1},y_{2}] \to
          \mathscr{G}X([y_{1},y_{2}])\setminus [x_{1},x_{2}]$ and $Y_{0}: \mathscr{G}X([y_{1},y_{2}])\setminus
          [x_{1},x_{2}]\to \mathscr{G}Y([x_{1},x_{2}])\setminus [y_{1},y_{2}]$ are conformal and inverse
          to one another.
        \item \label{se_fp}   $X_{0}(\mathbb{C})\subset \mathbb{C}\setminus [x_{3},x_{4}]$ and
          $Y_{0}(\mathbb{C})\subset \mathbb{C}\setminus [y_{3},y_{4}]$.
   \end{enumerate}
   \end{lem}

The proof of Lemma~\ref{further_properties} is done in Part 5.3 of
\cite{FIM} for $z=1/4$; it can be generalized directly up to
$z\in]0,1/4[$.

   \begin{lem}
   \label{me_fu}
   The function $\widetilde{w}$ defined in (\ref{expression_CGFy}) is meromorphic on
   $\mathbb{C}\setminus [y_{3},y_{4}]$ and has there a single pole, which is at $x_{3}$.
   \end{lem}

   \begin{proof}
   $X_{0}$ being meromorphic on $\mathbb{C}\setminus ([y_{1},y_{2}]\cup [y_{3},y_{4}])$
   and $w$ on $\mathbb{C}\setminus [x_{3},x_{4}]$, the function $\widetilde{w}$ defined in
   (\ref{expression_CGFy}) is \text{a priori} meromorphic on $\mathbb{C}\setminus
   ([y_{1},y_{2}]\cup [y_{3},y_{4}]\cup X_{0}^{-1}([x_{3},x_{4}]))$.
   But on the one hand, with \ref{se_fp} of Lemma~\ref{further_properties}, $X_{0}^{-1}([x_{3},x_{4}])=\emptyset$.
   And on the other hand, thanks to the gluing property of the CGF $w$, $\widetilde{w}$ satisfies
   $\widetilde{w}^{+}(t)=\widetilde{w}^{-}(t)$ for $t\in[y_{1},y_{2}]$, \text{i.e.}\ the limits of $\widetilde{w}(u)$
   when $u\to t\in[y_{1},y_{2}]$ from the upper and lower sides of the cut are equal. $\widetilde{w}$ is thus also
   meromorphic in a neighborhood of $[y_{1},y_{2}]$, see e.g.\ Parts~5.2.3 and 5.2.4 of \cite{FIM}.
   Finally, $\widetilde{w}$ is meromorphic on $\mathbb{C}\setminus[y_{3},y_{4}]$.

   Moreover, with (\ref{expression_CGFy}) and since $w$ has on $\mathbb{C}\setminus [x_{3},x_{4}]$
   only one pole, which happens to be at $x_{2}$, the only poles of $\widetilde{w}$ are
   at the points $t$ where $X_{0}(t)=x_{2}$. It is then easy to verify, by a direct calculation,
   that $x_{3}$ is the only solution
   to the latter equation.
   \end{proof}

We prove now the following
proposition, which completes the proof of Theorem~\ref{explicit_integral}.
  \begin{prop}
  \label{cc}
  We have $w^{-1}(w([x_{1},x_{2}]))\setminus [x_{1},x_{2}]=\emptyset$ and
  $\widetilde{w}^{-1}(\widetilde{w}([y_{1},y_{2}]))\setminus [y_{1},y_{2}]=\emptyset$.
  \end{prop}

  \begin{proof}
   In order to prove the first identity,
 it is enough to  show that for any fixed $u\in[x_{1},x_{2}]$,
   the only solution in $t$ of $w(t)=w(u)$ is $t=u$.

  If $u\in[x_{1},x_{2}]$, then $w(u)\in\wp_{1,3}([-\omega_{1}/2,\omega_{1}/2])$, see
  Proposition~\ref{transformation_cycles_uniformization} and Equation~(\ref{expression_CGFx}).
  Thus once again with
  (\ref{expression_CGFx}), the
  equation $w(t)=w(u)$ can be interpreted as
  \begin{equation}
  \label{eqeq}
       \wp_{1,3}(\omega)=\wp_{1,3}(\Upsilon),
  \end{equation}
  where $\omega=x^{-1}(t)-[\omega_{1}+\omega_{2}]/2$ and
  $\Upsilon\in[-\omega_{1}/2,\omega_{1}/2]$.

  \text{A priori}, Equation~(\ref{eqeq})
  admits the solutions $\omega=\pm\Upsilon+k_{1}\omega_{1}+k_{3}\omega_{3}$, with
  $k_{1},k_{3}\in\mathbb{Z}$---see Corollary~3.10.8 in \cite{JS}, then remember that $\wp_{1,3}$ is even and periodic
  w.r.t.\ $\omega_{1},\omega_{3}$.
  But in our case, $\omega$ belongs to a restricted region, namely
  $]-\omega_{3},0]\times[-\omega_{1}/(2\imath),\omega_{1}/(2\imath)]$. Indeed, as already noted in this section,
  we have $x^{-1}(\mathbb{C}\cup\{\infty\})\subset [0,\omega_{2}/2]\times [0,\omega_{1}/\imath]$,
  so that thanks to Proposition~\ref{omega_3=(3/4)omega_2}, we have $x^{-1}(\mathbb{C}\cup\{\infty\})-[\omega_{1}+\omega_{2}]/2
  \subset]-\omega_{3},0]\times [-\omega_{1}/(2\imath),\omega_{1}/(2\imath)]$.
  In the latter restricted domain, the only solutions
  to~(\ref{eqeq}) are $\omega=\pm\Upsilon$.
  In particular, we get $x^{-1}(t)-[\omega_{1}+\omega_{2}]/2=\pm (x^{-1}(u)-[\omega_{1}+\omega_{2}]/2)$, which yields
  $x^{-1}(t)=(1\mp 1)[\omega_{1}+\omega_{2}]/2\pm x^{-1}(u)$.

  Then, taking the image of the previous equality trough $\wp$ and using Equation~(\ref{zgh}),
  we obtain $f(t)=\wp((1\mp 1)[\omega_{1}+\omega_{2}]/2\pm x^{-1}(u))$. Since $\wp$
  is periodic w.r.t.\ $\omega_{1},\omega_{2}$ and even, we conclude that  $f(t)=f(u)$
   and finally that $t=u$, since $f$ is one-to-one, see~(\ref{ffu}).

  %In Subsection~\ref{Implicit_FIM}, we saw that $x^{-1}(\mathbb{C}\cup\{\infty\})=
  %([0,\omega_{2}/2[\times [0,\omega_{1}/\imath[)\cup   ([0,\omega_{1}/2]+\omega_{2}/2)$. In
  %particular, with Proposition~\ref{omega_3=(3/4)omega_2}, we get  $x^{-1}(\mathbb{C}\cup\{\infty\})
  %\subset]-\omega_{3}+\omega_{2}/2,\omega_{2}/2]\times[0,\omega_{1}/\imath[$. But $\wp_{1,3}$
  %takes each value of $\mathbb{C}\cup\{\infty\}$ twice on the parallelogram $]-\omega_{3}+
  %\omega_{2}/2,\omega_{2}/2]\times[0,\omega_{1}/\imath[$ and with the proof of Lemma~\ref{lemma_g_2_g_3}, we have
  %$\wp_{1,3}([-\omega_{1}/2,0])=\wp_{1,3}([0,\omega_{1}/2])=w([x_{1},x_{2}])$, so that we get
  %$w^{-1}(w([x_{1},x_{2}]))\setminus [x_{1},x_{2}]=\emptyset$.

  The proof of the identity $w^{-1}(w([x_{1},x_{2}]))\setminus [x_{1},x_{2}]=\emptyset$
  is concluded.
  Similar reasoning and the use of Equation~(\ref{expression_CGFy}) yield that
  $\widetilde{w}^{-1}(\widetilde{w}([y_{1},y_{2}]))\setminus [y_{1},y_{2}]=\emptyset$.
\end{proof}

\subsection{Proof of Theorem~7}
%\label{Proofss}

%We conclude Section~\ref{Study_CGF} by proving Theorem~\ref{w_w_tilde_algebraic}.
%\begin{proof}[Proof of Theorem~\ref{w_w_tilde_algebraic}]

   Our aim is to show that the function defined in (\ref{expression_CGFx}) is
   the only function having a pole at $x_{2}$ and solution to
   (\ref{eq_w_w_tilde_algebraic}).

\medskip

Denote $\omega_{4}=\omega_{2}/4$ and let $\wp_{1,4}$ be the
Weierstrass elliptic function with periods $\omega_{1},\omega_{4}$.
We recall that $\wp$ and $\wp_{1,3}$ are the Weierstrass elliptic
functions with respective periods $\omega_{1},\omega_{2}$ and
$\omega_{1},\omega_{3}$, where $\omega_3=3\omega_{2}/4$ thanks to
Proposition~\ref{www}.

To begin with, let us prove the following lemma.
\begin{lem}
\label{sjjj}
Let $\breve{\wp}$ be the Weierstrass elliptic function
with periods noted $\hat{\omega},\check{\omega}$ and let $n$ be a
positive integer. Then the Weierstrass elliptic function  with
periods $\hat{\omega}, \check{\omega}/n$ can be written in terms of
$\breve{\wp}$ as follows:
     \begin{equation}
     \label{link_several_wp}
      \breve{\wp}(\omega)+\sum_{k=1}^{n-1}\big[\breve{\wp}(\omega+k\check{\omega}/n)
          -\breve{\wp}(k\check{\omega}/n)\big].
     \end{equation}
\end{lem}
\begin{proof}
It is easy to verify that both the Weierstrass elliptic function having for
periods $\hat{\omega}, \check{\omega}/n$ and the function defined by
(\ref{link_several_wp}) satisfy the three properties hereafter:
they are elliptic with periods $\hat{\omega},\check{\omega}/n$;
they have only one pole in the fundamental parallelogram
$\hat{\omega}[0,1[+(\check{\omega}/n)[0,1[$, this pole is
at $0$ and is of order two;
they admit an expansion at $\omega=0$ equal to
$1/\omega^{2}+O(\omega^{2})$.
Therefore, they must coincide, see e.g.\ Part 8.10 on pages $227$--$230$ in \cite{LAW}.
\end{proof}

Now we notice that by applying the following addition formula
(see (6.8.10) on page $162$ in \cite{LAW})
     \begin{equation}
     \label{addition_theorem}
          \wp(\omega+\widetilde{\omega})=-\wp(\omega)
          -\wp(\widetilde{\omega})+\frac{1}{4}\left[
          \frac{\wp'(\omega)-\wp'(\widetilde{\omega})}
          {\wp(\omega)-\wp(\widetilde{\omega})}\right]^{2},
          \ \ \ \ \ \forall \omega,\widetilde{\omega},
     \end{equation}
to the Weierstrass elliptic function $\breve{\wp}$ in
(\ref{link_several_wp}) and by then using the identity
(\ref{relation_wp_wp'}), we can express the Weierstrass elliptic
function with periods $\hat{\omega},\check{\omega}/n$ as a rational
function of the Weierstrass elliptic function $\breve{\wp}$ with
periods $\hat{\omega},\check{\omega}$.

We shall apply this procedure in the proof of Lemmas~\ref{ll} and~\ref{lll}.

\begin{lem}
\label{ll}
We have
 \begin{equation}
     \label{equality_1214}
          \wp_{1,4}(\omega)=-2\wp(\omega)+
          \frac{\wp'(\omega)^{2}+\wp'(\omega_{2}/4)^{2}}
          {2[\wp(\omega)-\wp(\omega_{2}/4)]^{2}}+
          \frac{\wp'(\omega)^{2}}
          {4[\wp(\omega)-\wp(\omega_{2}/2)]^{2}}
          -\wp(\omega_{2}/2)-2\wp(\omega_{2}/4), \forall \omega
     \end{equation}
where $\wp(\omega_{2}/2)= f(x_{1})$, $\wp(\omega_{2}/4)=(1+4z^{2})/3$,
$\wp'(\omega_{2}/4)=-8z^{2}$ and where $\wp'(\omega)$ can be expressed
in terms of $\wp(\omega)$ and $z$ with Equations (\ref{relation_wp_wp'})
and (\ref{g_2_g_3}).

Furthermore,
     \begin{equation}
     \label{ff}
          \wp_{1,4}\big(\wp^{-1}(f(t))-
          [\omega_{1}+\omega_{2}]/2\big)=F(t),
          \ \ \ \ \ \forall t\in\mathbb{C},
     \end{equation}
where $F$ is defined in (\ref{def_F_F_tilde}) and $f$ in (\ref{ffu}).
\end{lem}

\begin{proof}
We have $\omega_4=\omega_2/4$ by definition of $\omega_4$. Then, with
(\ref{link_several_wp}), we can write
     \begin{equation*}
          \wp_{1,4}(\omega)=\wp(\omega)+\wp(\omega+\omega_{2}/2)
          +\wp(\omega+\omega_{2}/4)+\wp(\omega+3\omega_{2}/4)-
          \wp(\omega_{2}/2)-\wp(\omega_{2}/4)-\wp(3\omega_{2}/4).
     \end{equation*}
Using then addition formula (\ref{addition_theorem}) for $\wp$
as well as the three equalities $\wp(\omega_{2}/4)=\wp(3\omega_{2}/4)$,
$\wp'(\omega_{2}/4)=-\wp'(3\omega_{2}/4)$ and $\wp'(\omega_{2}/2)=0$---obtained from the facts that $\wp(\omega_{2}/2+\omega)$ is even
and $\wp'(\omega_{2}/2+\omega)$ is odd, see (6.8.12) on page 162 in \cite{LAW}---we get (\ref{equality_1214}).

Using the formula below (see e.g.\ Exercise $8$ on page $182$ in \cite{LAW})
     \begin{equation}
     \label{anti_duplication}
          \wp( \omega_{2}/4 )=\wp( \omega_{2}/2 )
          +\big[(\wp( \omega_{2}/2)-\wp(
          \omega_{1}/2))(\wp( \omega_{2}/2)-\wp(
          [\omega_{1}+\omega_{2}]/2))\big]^{1/2}
     \end{equation}
as well as $\wp(\omega_{1}/2) =f(x_{3})$,
$\wp([\omega_{1}+\omega_{2}]/2)=f(x_{2})$ and
$\wp(\omega_{2}/2)=f(x_{1})$, see the proof of Lemma~\ref{lemma_g_2_g_3},
we immediately find
 $\wp(\omega_{2}/4)=(1+4z^{2})/3$. With
(\ref{relation_wp_wp'}) and (\ref{g_2_g_3}), we derive
$\wp'(\omega_{2}/4)^{2}=64z^{4}$. Since $\wp$ is decreasing on
$]0,\omega_{2}/2[$, see e.g.\ Part 6.11 on pages 166--167 in \cite{LAW}, we have $\wp'(\omega_{2}/4)<0$
and therefore $\wp'(\omega_{2}/4)=-8z^{2}$.

Formula (\ref{equality_1214}) with the known values of
$\wp(\omega_2/2)$, $\wp(\omega_2/4)$, $\wp'(\omega_2/4)$ as well as
with~$\wp'(\omega)$ expressed in terms of $\wp(\omega)$ and $z$ thanks to
(\ref{relation_wp_wp'}) and (\ref{g_2_g_3}) gives a representation
of $\wp_{1,4}(\omega)$
%in terms of $z$
  as a rational function of
$\wp(\omega)$.

 Evaluating this representation at $\omega=\wp^{-1}(f(t))-[\omega_{1}+\omega_{2}]/2$,
 once again using (\ref{addition_theorem}) for the function $\wp$
 together with (\ref{relation_wp_wp'}) and (\ref{g_2_g_3}) for the derivatives as well as
 the explicit value of $\wp([\omega_{1}+\omega_{2}]/2)$ given above,
 we get (\ref{ff}), after a substantial but elementary calculation.
\end{proof}

%The proof of Theorem~\ref{w_w_tilde_algebraic} will follow from applying this fact twice:
%\begin{enumerate}
%\item \label{fi_4} first, since $\omega_{4}=\omega_{2}/4$, we will express $\wp_{1,4}$ as a rational function of $\wp$;
%\item \label{se_4} then, since $\omega_{4}=\omega_{3}/3$, we will express $\wp_{1,4}$ as a rational function of $\wp_{1,3}$.
%\end{enumerate}

%Before making explicit the rational transformations that appear with \ref{fi_4} and \ref{se_4},
%we explain how to conclude the proof of Theorem~\ref{w_w_tilde_algebraic}.

%An immediate
%consequence of \ref{fi_4} and \ref{se_4} is the possibility of writing $\wp_{1,3}$ as an
%algebraic function of $\wp$. In particular, it is clear from that and from the addition
%formula (\ref{addition_theorem}) for $\wp$ that the identity $w(t)=\wp_{1,3}\big(\wp^{-1}(f(t))-
%[\omega_{1}+\omega_{2}]/2\big)$, with $f(t)=d'(x_{4})/(t-x_{4})+d''(x_{4})/6$---which is the
%CGF under consideration, see (\ref{uniformization}) and (\ref{expression_CGFx})---defines an
%algebraic function of $t$.

\begin{lem}
\label{lll} We have
     \begin{equation}
     \label{dop}
           \wp_{1,4}(\omega)=-\wp_{1,3}(\omega)+\frac{\wp_{1,3}'(\omega)^2+
          \wp_{1,3}'(\omega_3/3)^2 }{ 2[ \wp_{1,3}(\omega)-\wp_{1,3}(\omega_3/3)]^2}
           -4\wp_{1,3}(\omega_3/3),
           \ \ \ \ \ \forall \omega.
     \end{equation}
\end{lem}
\begin{proof}
Formulas (\ref{link_several_wp}) and (\ref{addition_theorem})
combined with the fact that $\omega_{4}=\omega_{2}/4=\omega_{3}/3$,
see Proposition~\ref{www}, easily lead to (\ref{dop}).
\end{proof}

Equality (\ref{relation_wp_wp'}) for $\wp_{1,3}$, written as
\begin{equation}
\label{sss} \wp_{1,3}'(\omega)^{2}=4\wp_{1,3}(\omega)^{3}-
g_{2,1,3}\wp_{1,3}(\omega)-g_{3,1,3},
\end{equation}
  allows us to
express $\wp_{1,3}'(\omega)^{2}$ in terms of $\wp_{1,3}(\omega)$ and
invariants $g_{2,1,3},g_{3,1,3}$ associated with $\wp_{1,3}$.
The next lemma gives their expression in terms of $z$.
\begin{lem}
\label{g213g313}
 Invariants $g_{2,1,3},g_{3,1,3}$ of $\wp_{1,3}$
have the following explicit expressions:
     \begin{equation*}
          g_{2,1,3}=40\wp_{1,3}(\omega_3/3)^2/3-G_{2},
          \ \ \ \ \ g_{3,1,3}=-280 \wp_{1,3}(\omega_3/3)^3/27+14
          \wp_{1,3}(\omega_3/3)G_{2}/9+G_{3},
     \end{equation*}
where $G_{2},G_{3}$ are defined in (\ref{GGG}).
\end{lem}

\begin{proof}
The proof consists in expanding $\wp_{1,4}(\omega)$ at $\omega = 0$ in two different ways.

Firstly, we use equality (\ref{equality_1214}) with $\wp'(\omega)$
expressed in terms of $\wp(\omega)$ thanks to (\ref{relation_wp_wp'}), with $g_{2},g_{3}$
obtained in (\ref{g_2_g_3}) as well as with  $\wp(\omega_2/2)$, $\wp(\omega_2/4)$ and $\wp'(\omega_2/4)$
found in Lemma~\ref{ll}. Expanding this identity in a neighborhood of $\omega=0$, we obtain:
     \begin{equation}
     \label{first_expansion}
          \wp_{1,4}(\omega)=\frac{1}{\omega^{2}}+\big[9G_{2}/20\big]\omega^{2}-
          \big[27G_{3}/28\big]\omega^{4}+O(\omega^{6}).
     \end{equation}

Secondly, we expand $\wp_{1,4}(\omega)$ at  $\omega = 0$ using
Equation (\ref{dop}) with  $\wp'_{1,3}(\omega)$ and $\wp'_{1,3}(\omega_{3}/3)$
expressed thanks to (\ref{sss}). After some calculation, we get:
     \begin{align}
          \wp_{1,4}(\omega)=\frac{1}{\omega^{2}}
          &+\big[6
          \wp_{1,3}(\omega_3/3)^{2}-9g_{2,1,3}/20\big]\omega^{2} \label{second_expansion}\\
          &+
          \big[10\wp_{1,3}(\omega_3/3)^3-3
          \wp_{1,3}(\omega_3/3)g_{2,1,3}/2-27g_{3,1,3}/28\big]\omega^{4}+O(\omega^{6}).\nonumber
     \end{align}

Lemma~\ref{g213g313}  follows then as we identify the
expansions (\ref{first_expansion}) and (\ref{second_expansion}).
\end{proof}

In the next lemma, we compute $\wp_{1,3}(\omega_3/3)$.
\begin{lem}
\label{K_algebraic} We have
 $\wp_{1,3}(\omega_{3}/3)=K$,
 where $K$ is  found as the only real positive solution to Equation~(\ref{KKK}).
\end{lem}

\begin{proof}
It is stated in Exercise $7$ on page $182$ in \cite{LAW} that
the quantity $\wp_{1,3}(\omega_{3}/3)$ is the only real
positive solution to
$K^{4}-g_{2,1,3}K^{2}/2-g_{3,1,3}K-g_{2,1,3}^{2}/48=0$. Replacing
$g_{2,1,3}$ and $g_{3,1,3}$ with their expression obtained in
Lemma~\ref{g213g313}, we conclude that $\wp_{1,3}(\omega_{3}/3)$ is a root of
     \begin{equation}
     \label{eq_under_consideration}
          K^{4}-G_{2}K^{2}/2-G_{3}K-G_{2}^{2}/48.
     \end{equation}

Let us now show that for any $z\in]0,1/4[$, the polynomial
(\ref{eq_under_consideration}) has a unique real positive root.
Since $\wp_{1,3}(\omega_{3}/3)>0$---indeed, $\wp_{1,3}$ is positive
on $[0,\omega_{3}]$, see Part 6.11 on pages 166--167 in
\cite{LAW}---$\wp_{1,3}(\omega_{3}/3)$ shall be
characterized as the only real positive solution to~(\ref{KKK}).

According to Lemma~\ref{unique_G2_G3}, it is now enough to verify
that $G_{2}\neq 0$ and that $G_{2}^{3}-27G_{3}^{2}>0$. The first
fact is actually an immediate consequence of (\ref{GGG}), while the
second comes from the identity
$G_{2}^{3}-27G_{3}^{2}=(4^{14}/3^{6})z^{2}(z-1/4)^{4}(z+1/4)^{4}$,
see also (\ref{GGG}).
%According to Lemma~\ref{unique_G2_G3} below, it is enough to verify that
%$G_{2}^{3}-27G_{3}^{2}>0$; this fact comes the identity
%$G_{2}^{3}-27G_{3}^{2}=(4^{14}/3^{6})z^{2}(z-1/4)^{4}(z+1/4)^{4}$,
%obtained from~(\ref{GGG}).
\end{proof}

\begin{lem}
\label{unique_G2_G3}
For any real numbers $G_{2},G_{3}$ such that $G_{2}\neq 0$ and
$G_{2}^{3}-27G_{3}^{2}>0$, the polynomial (\ref{eq_under_consideration})
has a real negative root, a real positive root and two non-real
complex conjugate roots.
\end{lem}

\begin{proof}
If some real numbers $G_{2},G_{3}$ such that $G_{2}\neq 0$ and
$G_{2}^{3}-27G_{3}^{2}\neq 0$ are given, then there exists a lattice
$\hat{\omega}\mathbb{Z}+\check{\omega}\mathbb{Z}$ and a Weierstrass
elliptic function w.r.t.\ this lattice, say $\breve{\wp}$, having
for invariants $G_{2},G_{3}$, see Corollary 6.5.8 on page 287 in \cite{JS}.
By using once again Exercise $7$ on page 182 in \cite{LAW}, we then come to the conclusion that
$\breve{\wp}(\hat{\omega}/3),\breve{\wp}(\check{\omega}/3)$,
$\breve{\wp}([\hat{\omega}-\check{\omega}]/3),\breve{\wp}([\hat{\omega}+\check{\omega}]/3)$
are the four roots of the polynomial (\ref{eq_under_consideration}).

Now we prove that the inequality $G_{2}^{3}-27G_{3}^{2}>0$
yields that one of $\breve{\wp}(\hat{\omega}/3),\breve{\wp}(\check{\omega}/3)$ is negative while
the other is positive, and that $\breve{\wp}([\hat{\omega}-\check{\omega}]/3),\breve{\wp}([\hat{\omega}+\check{\omega}]/3)$
are complex conjugate to one another.

%Now we prove that the inequality $G_{2}^{3}-27G_{3}^{2}>0$ (resp.\ $G_{2}^{3}-27G_{3}^{2}<0$),
%implies that one of $\breve{\wp}(\hat{\omega}/3),\breve{\wp}(\check{\omega}/3)$ (resp.\ %$\breve{\wp}([\hat{\omega}-\check{\omega}]/3),\breve{\wp}([\hat{\omega}+\check{\omega}]/3)$) is negative while
%the other is positive and $\breve{\wp}([\hat{\omega}-\check{\omega}]/3),\breve{\wp}([\hat{\omega}+\check{\omega}]/3)$
%(resp.\ $\breve{\wp}(\hat{\omega}/3),\breve{\wp}(\check{\omega}/3)$) are complex conjugate the one another.

If $G_{2}^{3}-27G_{3}^{2}>0$,
then the periods $\hat{\omega},\check{\omega}$ of $\breve{\wp}$ can be chosen such that $\hat{\omega}>0$
and $\check{\omega}/\imath>0$. Indeed, with pages 110--111 (particularly Theorem~3.6.12) of \cite{JS},
we conclude that $\hat{\omega},\check{\omega}$ are either
real and purely imaginary (if the polynomial $4x^{3}-G_{2}x-G_{3}$ has three real roots) or complex conjugate to one another
(if $4x^{3}-G_{2}x-G_{3}$ has only one real root); in addition, it is well known that $4x^{3}-G_{2}x-G_{3}$ has three real roots
if and only if $G_{2}^{3}-27G_{3}^{2}>0$ and only one real root if and only if $G_{2}^{3}-27G_{3}^{2}<0$. Moreover,  on the
parallelogram $[0,\hat{\omega}]\times [-\check{\omega}/(2\imath),\check{\omega}/(2\imath)]$, $\breve{\wp}$ takes real values
on the segments $[0,\hat{\omega}]$, $[-\check{\omega}/2,\check{\omega}/2]$,
$[0,\hat{\omega}]\pm\check{\omega}/2$, $[-\check{\omega}/2,\check{\omega}/2]+\hat{\omega}/2$, and
$[-\check{\omega}/2,\check{\omega}/2]+\hat{\omega}$
and on those segments only, see Part~3.16 on pages $109$--$115$ in~\cite{JS}.

\begin{figure}[!ht]
\begin{center}
\begin{picture}(-100.00,715.00)
\hspace{-100mm}\includegraphics{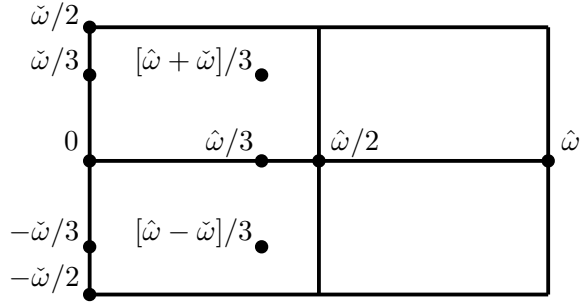}
\end{picture}
\end{center}
\vspace{-213mm}
\caption{The parallelogram $[0,\hat{\omega}]\times[-\check{\omega}/(2\imath),\check{\omega}/(2\imath)]$}
\label{real_values}
\end{figure}

As a first consequence, $\breve{\wp}(\hat{\omega}/3)$ and $\breve{\wp}(\check{\omega}/3)$
are real. Furthermore, on the segments $[0,\hat{\omega}]$ and $[-\check{\omega}/2,\check{\omega}/2]$,
the values of $\breve{\wp}$ are positive and negative respectively, see Part 6.11 on pages $166$--$167$ in \cite{LAW}. This is why $\breve{\wp}(\hat{\omega}/3)>0$ and $\breve{\wp}(\check{\omega}/3)<0$.

As a second consequence, $\breve{\wp}([\hat{\omega}-\check{\omega}]/3)$ and $\breve{\wp}([\hat{\omega}+\check{\omega}]/3)$
are non-real. Since $[\hat{\omega}-\check{\omega}]/3$ and $[\hat{\omega}+\check{\omega}]/3$ are
complex conjugate to one another, so are $\breve{\wp}([\hat{\omega}-\check{\omega}]/3)$ and $\breve{\wp}([\hat{\omega}+\check{\omega}]/3)$,
see also Part 6.11 on pages $166$--$167$ in \cite{LAW}. Lemma~\ref{unique_G2_G3} is proved.
\end{proof}

\begin{rem}
{\rm Lemma~\ref{unique_G2_G3} is also valid for any real numbers $G_{2},G_{3}$ such that
$G_{2}\neq 0$ and $G_{2}^{3}-27G_{3}^{2}<0$.
On the other hand, if $G_{2}^{3}-27G_{3}^{2}=0$, then the polynomial (\ref{eq_under_consideration})
has only real roots.}
\end{rem}

Substituting in (\ref{dop}) equality (\ref{sss}) for
$\wp'_{1,3}(\omega)$ and $\wp'_{1,3}(\omega_3/3)$, we express
$\wp_{1,4}(\omega)$ in terms of $\wp_{1,3}(\omega)$, $g_{2,1,3}$,
$g_{3,1,3}$ and $\wp_{1,3}(\omega_3/3)$. Applying then
Lemmas~\ref{g213g313} and~\ref{K_algebraic}, we get:
\begin{eqnarray*}
            \wp_{1,3}(\omega)^{3}
            -\wp_{1,3}(\omega)^{2}\big[\wp_{1,4}(\omega)+2K\big]\hspace{-2mm}
            &+&\hspace{-2mm}\wp_{1,3}(\omega)\big[2K\wp_{1,4}(\omega)+K^{2}/3+G_{2}/2\big]\\
            \hspace{-2mm}&-&\hspace{-2mm}\big[K^{2}\wp_{1,4}(\omega)+19G_{2}K/18+
            G_{3}-46K^{3}/27\big]=0.
       \end{eqnarray*}
In particular, evaluating this identity at
$\omega=\wp^{-1}(f(t))-[\omega_{1}+\omega_{2}]/2$, using (\ref{ff})
and taking into account (\ref{zgh}), we obtain Equation
(\ref{eq_w_w_tilde_algebraic}) for the CGF $w$ defined in (\ref{expression_CGFx}).

\medskip

If $F$, defined in (\ref{def_F_F_tilde}), is infinite at some point $t_{0}$, then
Equation~(\ref{eq_w_w_tilde_algebraic}) yields $[w(t_{0})-K]^{2}=0$. Thus $K$
is a double root of (\ref{eq_w_w_tilde_algebraic}) at $t_{0}$. In
addition,  the double root
   $K$ being non-zero by its definition via (\ref{KKK}) and
the product $K^{2}F(t_{0})+19G_{2}K/18+
            G_{3}-46K^{3}/27$ of all roots of the polynomial (\ref{eq_w_w_tilde_algebraic})
being infinite, the third root of (\ref{eq_w_w_tilde_algebraic}) must be
infinite.

Since $F$ is infinite at $x_{2}$, see (\ref{def_F_F_tilde}), and
since $w$ has a pole at $x_{2}$, see Subsection~\ref{Implicit_FIM},
 $w$
can thus be characterized as the unique solution to
(\ref{eq_w_w_tilde_algebraic}) with a pole at $x_{2}$---the two
other solutions are finite at $x_{2}$ and equal to $K$. The part of
Theorem~\ref{w_w_tilde_algebraic} dealing with a CGF for the set
$\mathscr{G}X([y_{1},y_{2}])$ is concluded.

\medskip

Now we prove the corresponding fact for
$\mathscr{G}Y([x_{1},x_{2}])$. Since $w$ is a solution to
(\ref{eq_w_w_tilde_algebraic}) and since an easy calculation gives
$F(X_{0}(t))=\widetilde{F}(t)$, we conclude that $\widetilde{w}$
satisfies the equation obtained from ({\ref{eq_w_w_tilde_algebraic}})
by replacing $F(t)$
by $\widetilde{F}(t)$. Furthermore, both $\widetilde{F}$ and
$\widetilde{w}$ have a pole at $x_{3}$: for $\widetilde{F}$, this is
a consequence of (\ref{def_F_F_tilde}) and for $\widetilde{w}$, this
follows from Lemma~\ref{me_fu}. Using  the same arguments as
above for $w$, we then derive that $\widetilde{w}$ can be characterized
as the only  function having a pole at $x_{3}$ and solution to
the equation obtained from ({\ref{eq_w_w_tilde_algebraic}}) by replacing $F(t)$
by $\widetilde{F}(t)$.
%\end{proof}
$\hfill \square$

\section{Holomorphic continuation of {\boldmath$zQ(x,0,z)$, $z(y+1)Q(0,y,z)$}}
\label{CONTINUATION}

In this part, we shall prove Theorem~\ref{Thm_continuation}. In other words,
we shall show that $zQ(x,0,z)$ and $z(y+1)Q(0,y,z)$ can be holomorphically
continued from their unit disc up to $\mathbb{C}\setminus [x_{3},x_{4}]$ and
$\mathbb{C}\setminus [y_{3},y_{4}]$ respectively.

%In fact, we are going to show that $Q(x,0,z)$ and $Q(0,y,z)$ can be holomorphically
%continued up to $\mathbb{C}\setminus [x_{3},x_{4}]$ and $\mathbb{C}\setminus [y_{3},y_{4}]$
%respectively; this assertion is equivalent, as proved at the end of the proof of
%Theorem~\ref{thm_continuation}.

\begin{proof}[Proof of Theorem~\ref{Thm_continuation}]
First of all, we lift the functions $Q(x,0,z)$ and $Q(0,y,z)$
up to $\mathbb{C}/\Omega$ by setting
$q_{x}(\omega)=Q(x(\omega),0,z)$
                and $q_{y}(\omega)=Q(0,y(\omega),z)$.
                We recall that $x(\omega)$ and $y(\omega)$,
                the coordinates of the uniformization, are
                defined in (\ref{uniformization}).
                The functions
                $q_{x}$ and $q_{y}$ are \text{a priori} well defined on
                $x^{-1}(\{|x|\leq 1\})$ and $y^{-1}(\{|y|\leq 1\})$ respectively.
                Then, we use the following result, that we shall prove in a few lines.
                \begin{thm}
                \label{thm_continuation}
                  $q_{x}(\omega)$ and $(y(\omega)+1)q_{y}(\omega)$, initially well defined on
                   $x^{-1}(\{|x|\leq 1\})$ and $y^{-1}(\{|y|\leq 1\})$ respectively,
                     can be holomorphically continued up to the whole parallelogram
                     $\mathbb{C}/\Omega$ cut along $[0,\omega_{1}[$ and
                     $[0,\omega_{1}[+\omega_{3}/2$ respectively. Moreover, these continuations satisfy
                          \begin{equation}
                          \label{in_view_of_projecting}
                               q_{x}(\omega)=q_{x}(\psi(\omega)),
                               \ \forall \omega\in\mathbb{C}/\Omega
                               \setminus [0,\omega_{1}[,
                               \ \ \ \ \ q_{y}(\omega)=q_{y}(\phi(\omega)),
                               \ \forall \omega\in\mathbb{C}/\Omega
                               \setminus ([0,\omega_{1}[+\omega_{3}/2),
                          \end{equation}
                     and
                          \begin{equation}
                          \label{continuation_functional_equation}
                               zq_{x}(\omega)+z(y(\omega)+1)q_{y}(\omega)
                               -zQ(0,0,z)-x(\omega)y(\omega)=0,
                               \ \forall \omega\in]3\omega_{2}/8,\omega_{2}[\times
                               [0,\omega_{1}/\imath[.
                          \end{equation}
                \end{thm}
%\begin{rem}
%{Both (\ref{ooo}) and (\ref{ooo1}) are immediate consequences
%of (\ref{continuation_functional_equation}).}
%\end{rem}
 Finally, we set $Q(x,0,z)=q_{x}(\omega)$ if $x(\omega)=x$ as well as
                $Q(0,y,z)=q_{y}(\omega)$ if $y(\omega)=y$. Thanks to (\ref{in_view_of_projecting})
                and Proposition~\ref{transformation_cycles_uniformization},
                these equalities define not ambiguously $Q(x,0,z)$ and $Q(0,y,z)$ on
                $\mathbb{C}\setminus [x_{3},x_{4}]$ and $\mathbb{C}\setminus [y_{3},y_{4}]$
                respectively, as holomorphic functions.
                 Furthermore, the statements
                   (\ref{ooo}) and (\ref{ooo1}) are immediate consequences
of (\ref{continuation_functional_equation}).
\end{proof}

%Items \ref{step_lifting} and \ref{projecting} are straightforward.

The proof of Theorem~\ref{Thm_continuation} is now reduced to that
of Theorem~\ref{thm_continuation}. In order to carry out this
proof, we need to find the location of the cycles
$x^{-1}(\{|x|=1\})$ and $y^{-1}(\{|y|=1\})$ on $\mathbb{C}/\Omega$.
This is the subject of the following result, illustrated on Figure
\ref{Locations_of_the_cycles} below.

\begin{figure}[!ht]
\begin{center}
\begin{picture}(000.00,710.00)
\hspace{-100mm}\includegraphics{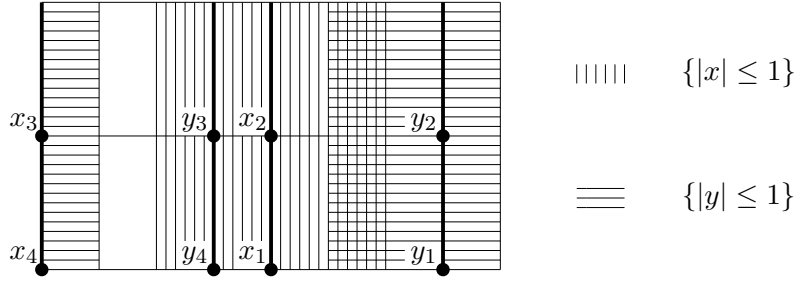}
\end{picture}
\end{center}
\vspace{-213mm}
\caption{Location of the important cycles on the surface $[0,\omega_{2}[\times[0,\omega_{1}/\imath[$}
\label{Locations_of_the_cycles}
\end{figure}

     \begin{prop}
     \label{transformation_circles_uniformization}
               We have $x^{-1}(\{|x|=1\})=([0,\omega_{1}[+\omega_{2}/4)\cup
               ([0,\omega_{1}[+3\omega_{2}/4)$ as well as
               $y^{-1}(\{|y|=1\})=([0,\omega_{1}[+\omega_{2}/8)\cup
               ([0,\omega_{1}[+5\omega_{2}/8)$.
     \end{prop}

\begin{proof}
     The details are of course essentially the same for $x$ and $y$, so
     that we are going to prove only the assertion dealing with $x$.

     First of all, we note that because of the equality $x\circ \psi=x$,
     it is sufficient to prove that
     $x^{-1}(\{|x|=1\})\cap ([0,\omega_{2}/2[\times [0,\omega_{1}/\imath[)=[0,\omega_{1}[+
     \omega_{2}/4$---the advantage of this is that $\wp$, and therefore also $x$, are
     one-to-one in the half-parallelogram $[0,\omega_{2}/2[\times [0,\omega_{1}/\imath[$.

     The proof is then composed of three steps.

     \medskip

\noindent{\it  Step 1.}  We prove that
$x(\omega_{2}/4+\omega_{1}/2)=1$. For this, we recall
     that $\wp(\omega_{2}/4)=(1+4z^{2})/3$, $\wp'(\omega_{2}/4)=-8z^{2}$, $\wp(\omega_{1}/2)
     =f(x_{3})$ with $f$ defined in (\ref{ffu})
      and $\wp'(\omega_{1}/2)=0$, see Lemma~\ref{ll} and
     its proof.  With addition formula (\ref{addition_theorem}), we then immediately
     obtain the explicit value of $\wp(\omega_{2}/4+\omega_{1}/2)$.
      Finally, by
     using Equation (\ref{uniformization}) and after a simple calculation,
      we get $x(\omega_{2}/4+\omega_{1}/2)=1$.

      \medskip

\noindent{\it Step 2.} We show that $x^{-1}(\{|x|=1\})\cap
([0,\omega_{2}/2[\times [0,\omega_{1}/\imath[)
     \subset[0,\omega_{1}[+\omega_{2}/4$.
      For this, let $\theta\in [0,2\pi[$ and consider
      the equation $x(\omega)=\exp(\imath \theta)$. Thanks to (\ref{uniformization}) and (\ref{ffu}),
     we obtain $\wp(\omega)=f(\exp(\imath \theta))$ and thus $\omega=\wp^{-1}(f(\exp(\imath \theta)))$.
     We can then use the explicit expression
     of the inverse function of $\wp$ on $[0,\omega_{2}/2[\times [0,\omega_{1}/\imath[$, see e.g.\ Part 6.12
     on pages 167--172 in \cite{LAW},
     %(6.12.3)
     and we get
          \begin{equation}
          \label{eq_with_reciprocal}
               \omega=x^{-1}(1)+
               \int_{f(1)}^{f(\exp(\imath \theta))}\frac{\text{d}t}
               {[4t^{3}-g_{2}t-g_{3}]^{1/2}}\\
               =\omega_{2}/4+\omega_{1}/2+\frac{1}{2}
               \int_{\exp(\imath \theta)}^{1}\frac{\text{d}x}{[d(x)]^{1/2}},
          \end{equation}
     where $d$ is defined in Section~\ref{reduc} and $g_{2},g_{3}$ in
     Lemma~\ref{lemma_g_2_g_3}.
     Note that the second equality above comes from the first step and from
     exactly the same calculations as in Part 3.3 of \cite{FIM}.
     Now we notice that $d(x)=x^{4}d(1/x)$.
     In particular, the change of variable $x\mapsto 1/x$ in the integral
     $\int_{\exp(\imath \theta)}^{1}\text{d}x/[d(x)]^{1/2}$ yields
     $\int_{\exp(\imath \theta)}^{1}\text{d}x/[d(x)]^{1/2}=-\int_{\exp(-\imath \theta)}^{1}
     \text{d}x/[d(x)]^{1/2}$. As a consequence, this integral belongs to $\imath \mathbb{R}$.
     In conclusion, with (\ref{eq_with_reciprocal}), we have shown that
     $x^{-1}(\{|x|=1\})\cap([0,\omega_{2}/2[\times$ $[0,\omega_{1}/\imath[)
     \subset[0,\omega_{1}[+\omega_{2}/4$.

     \medskip

\noindent{\it Step 3.} We prove that the inclusion above has to be
an equality. Indeed, if it was not
     the case, the curve $x^{-1}(\{|x|=1\})\cap ([0,\omega_{2}/2[\times [0,\omega_{1}/\imath[)$
     would be \textit{not} closed, which is a manifest contradiction with the facts that $\{|x|=1\}$
     is closed and that $x$ is meromorphic as well as one-to-one in the half-parallelogram
     $[0,\omega_{2}/2[\times [0,\omega_{1}/\imath[$.
\end{proof}

\begin{proof}[Proof of Theorem~\ref{thm_continuation}]
This proof is composed of two steps: at first, we shall
define the continuations of $q_{x}$ and $q_{y}$ on the
parallelogram $\mathbb{C}/\Omega$ cut along $[0,\omega_{1}[$ and
$[0,\omega_{1}[+\omega_{3}/2$ respectively; then, we shall
verify that these continuations satisfy the conclusions of
Theorem~\ref{thm_continuation}.

\medskip

\noindent{\it  Step 1.} We define the continuations of $q_{x}$ and
$q_{y}$.

\begin{enumerate}
\item\label{fi_tsp} We define $q_{x}(\omega)$ on $x^{-1}(\{|x|\leq 1\})$
                by $Q(x(\omega),0,z)$ and $q_{y}(\omega)$ on $y^{-1}
                (\{|y|\leq 1\})$ by $Q(0,y(\omega),z)$. Note that as a
                consequence of Proposition~\ref{transformation_circles_uniformization},
                we have $x^{-1}(\{|x|\leq 1\})=[\omega_{2}/4,3\omega_{2}/4]\times
                [0,\omega_{1}/\imath[$ and $y^{-1}(\{|y|\leq 1\})=[5\omega_{2}/8,9\omega_{2}/8]
                \times [0,\omega_{1}/\imath[$.

\item\label{se_tsp} Motivated by (\ref{functional_equation}), on
                $[3\omega_{2}/4,\omega_{2}[\times [0,\omega_{1}/\imath[
                \subset y^{-1}(\{|y|\leq 1\})$, we set $q_{x}(\omega)=
                -(y(\omega)+1)q_{y}(\omega)+Q(0,0,z)+x(\omega)y(\omega)/z$
                and on $]3\omega_{2}/8,5\omega_{2}/8]\times [0,\omega_{1}/\imath[
                \subset x^{-1}(\{|x|\leq 1\})$, we set $(y(\omega)+1)q_{y}(\omega)
                =-q_{x}(\omega)+Q(0,0,z)+x(\omega)y(\omega)/z$.

\item\label{th_tsp} On $]0,\omega_{2}/4]\times [0,\omega_{1}/\imath[$, we define
                $q_{x}(\omega)$ by $q_{x}(\phi(\omega))$. Note that with Equation
                (\ref{expression_automorphisms_torus}), we have $\phi(]0,\omega_{2}/4]
                \times [0,\omega_{1}/\imath[)=[3\omega_{2}/4,\omega_{2}[\times [0,\omega_{1}/\imath[$.
                On $[\omega_{2}/8,3\omega_{2}/8[\times [0,\omega_{1}/\imath[$, we define
                $q_{y}(\omega)$ by $q_{y}(\psi(\omega))$. By using (\ref{expression_automorphisms_torus}),
                we have $\psi([\omega_{2}/8,3\omega_{2}/8[\times [0,\omega_{1}/\imath[)=]3\omega_{2}/8,
                5\omega_{2}/8]\times [0,\omega_{1}/\imath[$.
\end{enumerate}
The functions $q_{x}$ and $q_{y}$ are now well defined on the whole parallelogram
$\mathbb{C}/\Omega$ cut along $[0,\omega_{1}[$ and $[0,\omega_{1}[+ \omega_{3}/2$
respectively.

%Let us note that the definition given in \ref{fi_tsp} above is quite natural. The one
%stated in \ref{se_tsp} is also natural since on $x^{-1}(\{|x|\leq 1\})\cap
%y^{-1}(\{|y|\leq 1\})=[5\omega_{2}/8,3\omega_{2}/4]\times [0,\omega_{1}/\imath[$,
%the equality $q_{x}(\omega)+(y(\omega)+1)q_{y}(\omega)-Q(0,0,z)-x(\omega)
%y(\omega)/z=0$ holds, see (\ref{functional_equation}). The definition set in
%\ref{th_tsp} is to ensure that (\ref{in_view_of_projecting}) is valid.

\medskip

\noindent\textit{Step 2}. We prove that the continuations of
$q_{x}$ and $q_{y}$ defined in the first step satisfy the different assertions of
Theorem~\ref{thm_continuation}.

  Let us verify Equation (\ref{in_view_of_projecting})
for $q_{x}$. By using \ref{fi_tsp} as well as the equality $x\circ
\psi=x$, (\ref{in_view_of_projecting}) is manifestly satisfied on
$[\omega_{2}/4,3\omega_{2}/4]\times
[0,\omega_{1}/\imath[=\psi([\omega_{2}/4,3\omega_{2}/4]\times
[0,\omega_{1}/\imath[)$. Moreover, with \ref{th_tsp},
(\ref{in_view_of_projecting}) is satisfied for $q_{x}$ on
$]0,\omega_{2}/4]\times [0,\omega_{1}/\imath[$. Since $\psi^{2}
=\text{id}$, (\ref{in_view_of_projecting}) is also true for $q_{x}$
on $[3\omega_{2}/4,\omega_{2} [\times [0,\omega_{1}/\imath[$ and thus
finally on the whole $\mathbb{C}/\Omega\setminus [0,\omega_{1}[$.
   Likewise, we easily verify  that (\ref{in_view_of_projecting}) is
valid for $q_{y}$ on $\mathbb{C}/\Omega  \setminus
([0,\omega_{1}[+3\omega_{2}/8)$. Equation
(\ref{continuation_functional_equation}) is immediately true, by
construction of the continuations.

   It remains to prove that the continuations of $q_{x}$ and $(y+1)q_{y}$
are holomorphic on $\mathbb{C}/\Omega$ cut along $[0,\omega_{1}[$
and $[0,\omega_{1}[+3\omega_{2}/8$ respectively.

     We first show that $q_x$ is {meromorphic} on its respective
 cut parallelogram. The following cycles are \text{a priori} problematic for $q_{x}$:
$[0,\omega_{1}[$, $[0,\omega_{1}[+\omega_{2}/4$ and
$[0,\omega_{1}[+3\omega_{2}/4$. In an open neighborhood of
$[0,\omega_{1}[+3\omega_{2}/4$, we have
$q_{x}(\omega)=-(y(\omega)+1)q_{y}(\omega)+
Q(0,0,z)+x(\omega)y(\omega)/z$, so that $q_{x}$ is in fact
meromorphic in the neighborhood of the cycle
$[0,\omega_{1}[+3\omega_{2}/4$. Since Equation
(\ref{in_view_of_projecting}) holds, $q_{x}$ is also meromorphic
near
$[0,\omega_{1}[+\omega_{2}/4=\psi([0,\omega_{1}[+3\omega_{2}/4)$.
Thus  $[0,\omega_{1}[$ remains the only singular cycle for $q_{x}$.
Furthermore, $q_{x}$ is clearly holomorphic on
$]\omega_{2}/4,3\omega_{2}/4]\times [0,\omega_{1}/\imath[$, since it
is  defined there through the power series $Q(x,0,z)$. On
$]5\omega_{2}/8,\omega_{2}[\times [0,\omega_{1}/\imath[$, we have
$q_{x}(\omega)=
-(y(\omega)+1)q_{y}(\omega)+Q(0,0,z)+x(\omega)y(\omega)/z$ and the
first two terms in the right-hand side of this equality are
holomorphic on this domain.
  With Lemma~\ref{some_particular_values} below,
  the product $x(\omega)y(\omega)$ may have a pole on this domain only
  at
$7 \omega_2/8$: in fact,
    $x$ has a pole of order one at this point but
     $y$ has there a zero of order two, so that the product
     $x(\omega) y(\omega)$ is holomorphic near
$7\omega_{2}/8$. On $]0,3\omega_{2}/8[\times [0,\omega_{1}/\imath[$,
we have $q_{x}=q_{x}\circ \psi$, so that $q_{x}$ is holomorphic on
this domain, since it is on $\psi(]0,3\omega_{2}/8[ \times
[0,\omega_{1}/\imath[)=]5\omega_{2}/8,\omega_{2}[\times
[0,\omega_{1}/\imath[$.
%If the cycle $[0,\omega_{1}[$ was not singular, the function $q_{x}(\omega)$ would
%be elliptic and $Q(x,0,z)$ rational, which is a priori not a problem.
%Note that the equation  $zq_{x}(\omega)+z(y(\omega)+1)q_{y}(\omega)-zQ(0,0,z)-x(\omega)
%y(\omega)=0$ holds on $[(3/8)\omega_{2},\omega_{2}]\times [0,\omega_{1}/\imath[$.

A similar reasoning yields that $(y+1)q_{y}$ is holomorphic on
$\mathbb{C}/\Omega \setminus ([0,\omega_{1}[+3\omega_{2}/8)$.
\end{proof}

\begin{cor}
The function $q_{y}$ is holomorphic on
$\mathbb{C}/\Omega \setminus ([0,\omega_{1}[+3\omega_{2}/8)$.
\end{cor}

\begin{proof}
From Theorem~\ref{thm_continuation}, we know that $(y+1)q_{y}$ is
holomorphic on $\mathbb{C}/\Omega \setminus
([0,\omega_{1}[+3\omega_{2}/8)$, so that we directly derive that $q_{y}$
is holomorphic on the same domain, except eventually at the points where $y+1=0$,
\text{i.e.}\ at $\omega_{2}/8$ and $5\omega_{2}/8$, see Lemma~\ref{some_particular_values}.
%Let us now to show that $q_{y}$ is
%holomorphic at these points, namely $\omega_{2}/8$ and
%$5\omega_{2}/8$, in accordance with Lemma
%\ref{some_particular_values}.
However, the generating function $Q(0,y,z)$ is bounded at $y=-1$,
see Section~\ref{reduc}, so that $q_{y}(\omega)=Q(0,y(\omega),z)$, being meromorphic
and bounded near $\omega_{2}/8$ and $5\omega_{2}/8$,
is actually holomorphic at both these points.
%We can also remark that with
%(\ref{continuation_functional_equation}), $(y(5\omega_{2}/8)
%+1)q_{y}(5\omega_{2}/8)=0$, since with Lemma
%\ref{some_particular_values}, $x(5\omega_{2}/8)=0$. Moreover, since
%$\phi(5\omega_{2}/8) =\omega_{2}/8$,
%$(y(\omega_{2}/8)+1)q_{y}(\omega_{2}/8)=0$. In other words, at
%$\omega=\omega_{2}/8$ and $\omega=5\omega_{2}/8$, both holomorphic
%functions $(y+1)q_{y}$ and $(y+1)$ have a zero, the first one of
%order equal or larger than one, the second one of order exactly one;
%it follows immediately that $q_{y}$ is holomorphic at these two
%values of $\omega$.
\end{proof}

The following lemma, which has been used in the proof of Theorem~\ref{thm_continuation},
easily follows from Lemma~\ref{Properties_X_Y_0} and from the fact that on the parallelogram
$[0,\omega_{2}[\times [0,\omega_{1}/\imath [$, the Weierstrass
elliptic function $\wp$ takes each value of
$\mathbb{C}\cup\{\infty\}$ twice.
     \begin{lem}
     \label{some_particular_values}
          The only poles of $x$ (of order one) are at $\omega_{2}/8,7\omega_{2}/8$
            and its only zeros (of order one) are at
          $3\omega_{2}/8,5\omega_{2}/8$. The only pole of $y$ (of order two)
           is at $3\omega_{2}/8$
          and its only zero (of order two) is at $7\omega_{2}/8$.
          The only zeros of $y+1$ are at $\omega_{2}/8,5\omega_{2}/8$.
     \end{lem}

     \noindent{\bf Acknowledgments.} We thank Philippe Bougerol for
drawing our attention to recent combinatorics problems and results.
Many thanks too to Amaury Lambert for suggesting we study Mireille
Bousquet-M\'elou's articles. We also thank Alexis Chommeloux and J\'er\'emie Lumbroso for their
numerous remarks about the English. Finally, we must thank an
anonymous referee for his/her very careful reading and his/her
valuable comments and suggestions.

\end{document}